\documentclass[11pt,a4paper,oneside,reqno]{amsart}

\usepackage{amsmath, amsthm, amsopn, amssymb}
\usepackage{mathtools, thmtools, thm-restate}
\usepackage{bbm}
\usepackage{enumerate}
\usepackage{enumitem}
\usepackage{hyperref}
\usepackage[capitalize]{cleveref}
\usepackage{multirow, makecell}
\usepackage[toc,page]{appendix}

\usepackage{stmaryrd} 
\usepackage{xcolor}

\usepackage{geometry}
\usepackage{setspace}
\declaretheorem[name=Theorem,within=section]{thm}
\declaretheorem[name=Fact,sibling=thm]{fact}
\declaretheorem[name=Corollary,sibling=thm]{cor}
\declaretheorem[name=Lemma,sibling=thm]{lemma}

\declaretheorem[name=Proposition,sibling=thm]{prop}
\declaretheorem[name=Claim,sibling=thm]{claim}

\crefname{thm}{theorem}{theorems}

\newcommand{\bG}{\mathbf{G}}
\newcommand{\bGs}{\mathbf{G^*}}
\newcommand{\bM}{\mathbf{M}}
\newcommand{\bR}{\mathbf{R}}
\newcommand{\bRs}{\mathbf{R^*}}

\newcommand{\NN}{\mathbb{N}}
\newcommand{\RR}{\mathbb{R}}

\newcommand{\cA}{\mathcal{A}}
\newcommand{\cB}{\mathcal{B}}
\newcommand{\cC}{\mathcal{C}}

\newcommand{\cF}{\mathcal{F}}
\newcommand{\cG}{\mathcal{G}}
\newcommand{\cH}{\mathcal{H}}
\newcommand{\cI}{\mathcal{I}}

\newcommand{\cM}{\mathcal{M}}

\newcommand{\cP}{\mathcal{P}}

\newcommand{\cS}{\mathcal{S}}

\newcommand{\cX}{\mathcal{X}}
\newcommand{\cY}{\mathcal{Y}}

\newcommand{\fC}{\mathfrak{C}}

\renewcommand{\Pr}{\mathbb{P}}
\newcommand{\Ex}{\mathbb{E}}

\newcommand{\1}{\mathbbm{1}} 

\newcommand{\Bin}{\mathrm{Bin}}
\newcommand{\Hyp}{\mathrm{Hyp}}

\newcommand{\Cop}{\mathrm{Cop}}

\newcommand{\Gnm}{G(n,m)}
\newcommand{\Gnmp}{G(n,m+1)}
\newcommand{\Gnmdp}{G(n,m+d+1)}
\newcommand{\Gnp}{G_{n,p}}

\newcommand{\eps}{\varepsilon}

\newcommand{\br}[1]{\llbracket{#1}\rrbracket}

\renewcommand{\le}{\leqslant}
\renewcommand{\ge}{\geqslant}

\newcommand{\deficit}{\mathrm{def}}
\newcommand{\ext}{\mathrm{ext}}
\renewcommand{\int}{\mathrm{int}}
\newcommand{\crit}{\mathrm{crit}}
\newcommand{\core}{\mathrm{core}}
\newcommand{\Core}{\mathrm{CORE}}

\newcommand{\eq}{\mathrm{eq}}

\hypersetup{
	colorlinks=true,
	linkcolor=blue,
	filecolor=magenta,      
	urlcolor=cyan
}

\author{Ilay Hoshen}
\address{School of Mathematical Sciences, Tel Aviv University, Tel Aviv 6997801, Israel}
\email{ilayhoshen@gmail.com}

\author{Wojciech Samotij}
\address{School of Mathematical Sciences, Tel Aviv University, Tel Aviv 6997801, Israel}
\email{samotij@tauex.tau.ac.il}

\author{Maksim Zhukovskii}
\address{Department of Computer Science, University of Sheffield, Sheffield S1 4DP, UK}
\email{m.zhukovskii@sheffield.ac.uk}

\thanks{This research was supported by: the Israel Science Foundation grant 2110/22; the grant 2019679 from the United States--Israel Binational Science Foundation (BSF) and the United States National Science Foundation (NSF); and the ERC Consolidator Grant 101044123 (RandomHypGra).}

\title{Stability of large cuts in random graphs}
\date{}
\begin{document}

\begin{abstract}
  We prove that the family of largest cuts in the binomial random graph exhibits the following stability property:
  If $1/n \ll p = 1-\Omega(1)$, then, with high probability, there is a set of $n - o(n)$ vertices that is partitioned in the same manner by all maximum cuts of $\Gnp$.
  Moreover, the analogous statement remains true when one replaces maximum cuts with nearly-maximum cuts.

  We then demonstrate how one can use this statement as a tool for showing that certain properties of $\Gnp$ that hold in a fixed balanced cut hold simultaneously in all maximum cuts.
  We provide two example applications of this tool.
  First, we prove that maximum cuts in $\Gnp$ typically partition the neighbourhood of every vertex into nearly equal parts; this resolves a conjecture of DeMarco and Kahn for all but a~narrow range of densities $p$.
  Second, for all edge-critical, nonbipartite, and strictly $2$-balanced graphs $H$, we prove a lower bound on the threshold density $p$ above which every largest $H$-free subgraph of $\Gnp$ is $(\chi(H)-1)$-partite.
  Our lower bound exactly matches the upper bound on this threshold recently obtained by the first two authors.
\end{abstract}

\maketitle

\section{Introduction}
\label{sec:introduction}

Let $r\ge 2$ be an integer.  An \emph{$r$-cut} in a graph $G$ is a partition of its vertex set into $r$ subsets. The \emph{size} of a cut is the number of edges of $G$ with endpoints in different parts of the cut.  A \emph{maximum $r$-cut} of $G$ is an $r$-cut that has the largest size (among all $r$-cuts in $G$). 
Since the problem of determining the size of a maximum $r$-cut in a graph
is NP-hard for every $r \ge 2$, much effort has been devoted to studying maximum cuts in random graphs, which can be viewed as the average-case variant of the general problem.
In spite of this, we still know relatively little even about maximum $2$-cuts in the binomial random graph $\Gnp$, except in the regime $p = O(1/n)$, see~\cite{CGHS04,DMS17,GL18,GKK18}.  In particular, the state-of-the-art result~\cite{C-OMS06} falls short of determining the precise asymptotics of the second-order term\footnote{It is easy to see that the first-order term is $(1-1/r)n^2p/2$.} in the typical size of a maximum $r$-cut in $\Gnp$ in the regime $p \gg 1/n$, determining it only up to a multiplicative constant.

The main obstacle in studying properties of maximum cuts in random graphs is that, even if a given property holds with high probability\footnote{With probability approaching 1 as $n\to\infty$; in what follows, we will write `whp' for brevity.} for a fixed cut, commonly this probability is not sufficiently close to one to warrant a union bound over all (exponentially many in the number of vertices) choices of a cut.
In an attempt to remedy this, Brightwell, Panagiotou, and Steger~\cite{BriPanSte12} proved that, if $1/n \ll p \le 1/2$, then whp any two maximum 2-cuts in $\Gnp$ differ in a small number of vertices.  Although this implies that the number of maximum 2-cuts in $\Gnp$ is typically subexponential in $n$, this is still not enough to treat a maximum cut as a fixed cut.  Brightwell, Panagiotou, and Steger also suggested that whp $\Gnp$ has the following stronger property:  There are disjoint vertex sets $S_1$ and $S_2$, each containing nearly $n/2$ vertices, that are contained in opposite parts of every maximum 2-cut in $G_{n, p}$.  Intuitively, such property could prove useful in approximating a maximum cut by a fixed cut.   This is because additionally requiring that $S_1$ and $S_2$ are inclusion-maximal with the above property ensures that the pair $\{S_1, S_2\}$ is unique.

In this paper, we not only confirm the prediction of Brightwell, Panagiotou, and Steger but also show how one may exploit this `clustering' property of maximum cuts while studying properties of $\Gnp$.
Let $b_r(G)$ denote the maximum size of an $r$-cut of $G$.  The \emph{deficit} of an $r$-cut in $G$ is the difference between $b_r(G)$ and the size of the cut.
Given integers $d \ge 0$ and $r \ge 2$ and a real $\alpha > 0$, we say that a graph $G$ \emph{admits a $(d,r,\alpha)$-core} if there exist $r$ pairwise-disjoint sets of vertices $S_1, \dotsc, S_r$ of size exceeding $(1/r-\alpha) \cdot |V(G)|$ each that are contained in different parts of every $r$-cut of $G$ with deficit at most $d$.
(We will give a precise definition of a core in \Cref{sec:rigidity}.)
The following statement is an abbreviated, qualitative version of our main result.
We postpone the full, unabbreviated  statement to \Cref{sec:rigidity}.

\begin{thm}
  \label{thm:main-rigidity-theorem}
  The following hold for every integer $r \ge 2$, real $\alpha > 0$, and all $p = p(n)$ satisfying $1/n \ll p = 1-\Omega(1)$:
  \begin{enumerate}[label=(\roman*)]
  \item
    \label{item:rigidity-1-statement}
    If $d \ll \sqrt{np}$, then whp $\Gnp$ admits a $(d,r,\alpha)$-core.
  \item
    \label{item:rigidity-0-statement}
    There exists a $C = C(\alpha) >0$ such that, for every $d \ge C\sqrt{np}$, whp no set of $\alpha n$ vertices is entirely contained in a single part of every $r$-cut with deficit at most~$d$.
  \end{enumerate}
\end{thm}

\Cref{thm:main-rigidity-theorem} implies, in particular, that whp each maximum $r$-cut in $\Gnp$ `respects' a single partition of all but some $o(n)$ vertices into $r$ parts.
The aforementioned unabbreviated version of the theorem replaces this $o(n)$ with $\omega\sqrt{n/p}$, where $\omega$ is an arbitrary function that tends to infinity with $n$.
While we do believe that this is optimal, in the sense that whp the largest set of vertices that is partitioned the same way by all largest $r$-cuts has size $n - \Omega(\sqrt{n/p})$, currently we cannot even show that the probability that $G_{n,p}$ has a unique largest $2$-cut is bounded away from one.
In order to support our belief, we will at least show that the \emph{expected} number of vertices of $G_{n,p}$ that lie outside of its core is $\Omega(\sqrt{n/p})$. We discuss this in more detail in \Cref{subsection:complement-of-the-core}.

An important aspect of \Cref{thm:main-rigidity-theorem} is that the property of admitting a core is `stable' under small perturbations of the edge set of a graph.  More precisely, if a graph $G$ admits a $(d,r,\alpha)$-core, then any graph $G'$ that is obtained from $G$ by adding/removing some $t \le d$ edges to/from $G$ still admits a $(d-t,r,\alpha)$-core (moreover, the core of $G'$ contains the core of $G$ in a sense that will be made precise by \Cref{cor:core-d-nested} below).
We will be able to exploit this fact when analysing certain properties of $\Gnp$ that concern the distribution of its edges relative to a maximum cut.  Roughly speaking, the stability property of cores will allow one to change a small proportion of adjacencies in $\Gnp$ while keeping the set of maximum cuts essentially unchanged (and thus determined up to the locations of $o(n)$ vertices).

In the second part of this work, we present two applications of \Cref{thm:main-rigidity-theorem}.  First, we resolve (almost fully) a conjecture of DeMarco and Kahn~\cite{DeMKah15Tur} stating that maximum cuts partition neighbourhoods of all vertices almost equally, see~\Cref{subsection:intro-neighbourhood} below.  Second, for every nonbipartite edge-critical and strictly $2$-balanced graph $H$, we prove an optimal lower bound on the threshold probability for the property that every largest $H$-free subgraph of $\Gnp$ is $(\chi(H)-1)$-partite; a matching upper bound on this threshold probability was recently proved by the first two authors~\cite{HosSam}.  This resolves another problem proposed by DeMarco and Kahn~\cite{DeMKah15Tur} and refutes their guess regarding the location of the said threshold.  We present the relevant background, as well as our results, in \Cref{subsection:intro-0-statement}.

\subsection{Neighbourhoods in maximum cuts}
\label{subsection:intro-neighbourhood}
It is easy to show that, as soon as $p\gg \log n/n$, a fixed $2$-cut of $G_{n, p}$ into sets of size $n/2 + o(n)$ (we call such a cut \emph{balanced}) whp divides the neighbourhood of every vertex into parts of size $(1/2+o(1))np$.  Obviously, this is not true for all balanced cuts, but is it true for all maximum cuts?  It was conjectured by DeMarco and Kahn~\cite[Conjecture 13.2]{DeMKah15Tur} that the answer is yes, for all $p \gg \log n/n$.  In \Cref{sec:neighbourhood}, we confirm the conjecture for all $p\gg (\log n)^2/n$.  In fact, with only little extra work, we generalise this result to $r$-cuts and common neighbourhoods of sets of vertices of any constant size.

\begin{thm}\label{thm:neighbourhood}
  Let $k \ge 1$, $r \ge 2$ be integers and let $\varepsilon>0$.  There exists $C>0$ such that, for all $p$ satisfying $p \ge C (\log n / n)^{1/k}$ and $p \gg (\log n)^2/n$, whp every maximum $r$-cut of $\Gnp$ partitions the common neighbourhood of every set of $k$ vertices into parts of size $(1/r \pm \eps)np^k$ each.
\end{thm}

It is not hard to see that the lower-bound assumption on $p$ is optimal (up to the value of $C$) for all $k \ge 2$ and $r \ge 2$.
While we do believe that the assumption $p \gg (\log n)^2/n$ (which comes in force only when $k = 1$) may be weakened to $p \gg \log n/n$,
proving this would likely involve significantly new ideas.
The reason why we need this additional assumption is that, in order to overcome the union bound over the choice of the set of vertices whose neighbourhood we are analysing, we need to resample $\Theta(\log n)$ edges of $\Gnp$; however, this seems to require good control over cuts with deficit $\Theta(\log n)$, which we do not have unless $p \gg (\log n)^2/n$, see the second part of \Cref{thm:main-rigidity-theorem}.

\subsection{Sharp thresholds for Simonovits's theorem in $\Gnp$}\label{subsection:intro-0-statement}

The well-known theorem of Tur\'an~\cite{Tur41} states that, for every $r \ge 2$ and all $n$, the largest $K_{r+1}$-free subgraphs of the complete graph $K_n$ are its largest $r$-partite subgraphs.
Simonovits~\cite{Sim68} proved that a similar result holds also for edge-critical graphs;
we say that a graph $H$ is \textit{edge-critical} if $\chi(H \setminus e) = \chi(H) - 1$ for some edge $e \in H$, where $\chi(H)$ is the chromatic number of $H$.

\begin{thm}[\cite{Sim68}]\label{thm:Simonovits}
  If a graph $H$ is edge-critical and $n$ is a sufficiently large integer, then every largest $H$-free subgraph of $K_n$ is $(\chi(H) - 1)$-partite.
\end{thm}

Let us call a graph \emph{$H$-Simonovits}, if each of its largest $H$-free subgraphs is $(\chi(H) - 1)$-partite.
Note that the assumption that $H$ is edge-critical in the above theorem is crucial.
Indeed, adding one edge to a $(\chi(H)-1)$-partite graph cannot introduce a copy of $H$ unless $H$ is edge-critical.
Consequently, if $H$ is not edge-critical, then no graph with chromatic number at least $\chi(H)$ can be $H$-Simonovits (in particular, no $K_n$ with $n\ge\chi(H)$ is $H$-Simonovits).

To the best of our knowledge, the first to study the question of when $\Gnp$ is whp $H$-Simonovits were Babai, Simonovits, and Spencer~\cite{BabSimSpe90}.  They proved that, for every $\ell \ge 1$, $G_{n, p}$ is whp $C_{2\ell+1}$-Simonovits as long as $p \ge 1/2 - \eps_\ell$ for some (small) positive constant $\eps_\ell$ that depends only on $\ell$.  Answering a challenge raised by the authors of \cite{BabSimSpe90}, Brightwell, Panagiotou, and Steger~\cite{BriPanSte12} proved that $G_{n, p}$ is whp $K_{r+1}$-Simonovits, for every $r \ge 2$, already when $p \ge n^{-c_r}$ for some (small) constant~$c_r>0$. 

It is not hard to see that as soon as the expected number of copies of some subgraph $F \subseteq H$ in $\Gnp$ becomes significantly smaller than the expected number of edges, $\Gnp$ cannot be $H$-Simonovits.\footnote{Unless $p = O(1/n)$, since then $\chi(\Gnp) = O(1)$ whp and a more careful analysis is required.}  This implies that $\Gnp$ cannot be $H$-Simonovits whp, unless $p=\Omega(n^{-1/m_2(H)})$, where (we write $e_F$ and $v_F$ for the numbers of edges and vertices of a~graph~$F$, respectively)
\[
  m_2(H) \coloneqq \max\left\{ \frac{e_F - 1}{v_F - 2} \colon F \subseteq H,\ e_F \ge 2\right\}
\]
is the \emph{$2$-density} of $H$.
Kohayakawa, {\L}uczak, and R\"odl~\cite{KohLucRod97} conjectured that, for every nonbipartite graph $H$ (not necessarily edge-critical), when $p \gg n^{-1 / m_2(H)}$, then whp every largest $H$-free subgraph of $G_{n, p}$ is close to being $(\chi(H) - 1)$-partite (i.e., it can be made $(\chi(H)-1)$-partite by removing some $o(n^2p)$ edges).  This was proved in the breakthrough work of Conlon and Gowers~\cite{ConGow16}, under the technical assumption that $H$ is also \emph{strictly $2$-balanced} (i.e., the maximum in the definition of the $2$-density is achieved uniquely at $F = H$), which was later removed by the second author~\cite{Sam14}, using an adaptation of the argument of Schacht~\cite{Sch16}, who proved the slightly weaker assertion that whp every $H$-free subgraph of $G_{n,p}$ has at most $(1-1/(\chi(H)-1)+o(1))\binom{n}{2}p$ edges.

In another major development, DeMarco and Kahn~\cite{DeMKah15Tur,DeMKah15Man} showed that adding an extra polylogarithmic factor in the lower bound on $p$ suffices to show that, when $H$ is a clique, then whp $\Gnp$ is $H$-Simonovits.  Moreover, their lower-bound assumption on $p$ is best possible up to a constant factor.  This result was recently generalised by the first two authors \cite{HosSam} to every nonbipartite, strictly $2$-balanced, and edge-critical graph $H$.

In order to state the main result of~\cite{HosSam}, we need an additional definition.
First, let us denote by $K_r(m)$ the complete, balanced $r$-partite graph with parts of size~$m$ and let $K_r(m)^+$ be the graph obtained from $K_r(m)$ by adding a single edge contained in one of the parts.
(Note that $H$ is edge-critical if and only if $H \subseteq K_{\chi(H) - 1}(m)^+$ for all $m \ge v_H$.)
Letting $\Cop(H, G)$ be  the number of copies of $H$ in $G$, set
\begin{equation}
  \label{eq:pi_H}
  \pi_H \coloneqq \lim_{m \to \infty} \frac{\Cop\big(H, K_{\chi(H)-1}(m)^+\big)}{m^{v_H-2}} > 0.
\end{equation}
Finally, let $\theta_H$ be the positive real satisfying
\begin{equation}
  \label{eq:theta_H}
  (\chi(H)-1)^{2-v_H} \cdot \pi_H \cdot \theta_H^{e_H-1} = 2 - \frac{1}{m_2(H)}.
\end{equation}

\begin{thm}[\cite{HosSam}]
  \label{thm:hoshen-samotij}
  If $H$ is a nonbipartite, edge-critical, strictly $2$-balanced graph and
  \begin{align*}        
    p \ge (\theta_H+\eps) \cdot n^{-1/m_2(H)} (\log n)^{1/(e_H - 1)}
  \end{align*}
  for some positive constant $\eps$, then whp $\Gnp$ is $H$-Simonovits.
\end{thm}

In \Cref{sec:0-statement}, we show that $\theta_H n^{-1/m_2(H)} (\log n)^{1/(e_H - 1)}$ is in fact a sharp threshold for the property of being $H$-Simonovits.  This confirms the prediction made in~\cite{HosSam} (and refutes the suggestion of DeMarco and Kahn~\cite{DeMKah15Tur} in the case of when $H$ is a complete graph).
\begin{thm}
  \label{thm:0-statement-Simonovits}
  If $H$ is a nonbipartite, edge-critical, strictly $2$-balanced graph and
  \begin{align*}        
    1/n \ll p \le (\theta_H-\eps) \cdot n^{-1/m_2(H)} (\log n)^{1/(e_H - 1)}
  \end{align*}
  for some positive constant $\eps$, then whp $\Gnp$ is not $H$-Simonovits.
\end{thm}

It remains an interesting open problem to extend \Cref{thm:hoshen-samotij,thm:0-statement-Simonovits} to edge-critical graphs $H$ that are not strictly $2$-balanced.  We note that the closely-related problem of describing the typical structure of $H$-free graphs with given order and size, for an arbitrary edge-critical graph $H$, was studied in~\cite{EngSamWar}.

\subsection{Relation to previous works}
\label{sec:relat-prev-works}

The notions of cores and rigidity (defined in
\Cref{sec:rigidity}), which play a central role in our arguments,
were first introduced by DeMarco and Kahn~\cite{DeMKah15Tur}.
In fact, one can infer from the proof of \cite[Lemma~12.3]{DeMKah15Tur} that the
random graph $\Gnp$ is typically close (up to adding/removing a small number of edges) to a graph that has a $(0,r,o(1))$-core.
Finally, even though the result of Brightwell, Panagiotou, and Steger~\cite{BriPanSte12}
does not imply the existence of a core, our proof of the fact that the random
graph typically has a $(0,r,o(1))$-core (the case $d=0$ in \Cref{thm:main-rigidity-theorem}) was inspired by their arguments.

\subsection*{Organisation}

The remainder of this paper is organised as follows.
\Cref{sec:rigidity} is devoted to the proof of \Cref{thm:main-rigidity-theorem}.
In fact, the main results of this section are \Cref{thm:rigidity} and \Cref{cor:core}, the latter of which is a stronger, non-asymptotic version of \Cref{thm:main-rigidity-theorem}.
In the following two sections, we prove  \Cref{thm:neighbourhood,thm:0-statement-Simonovits}, respectively.

\section{Rigidity}
\label{sec:rigidity}

Recall the definitions of an $r$-cut and the size and the deficit of a cut.
In order to quantify the `clustering' property of the family of cuts with a small deficit in a graph $G$, we will count pairs of vertices of $G$ that are never separated by such a cut.
To make this precise, given integers $r \ge 2$ and $d \ge 0$, we say that two vertices are \emph{$(d,r)$-equivalent} in $G$ if they are in the same part of every $r$-cut in $G$ with deficit at most $d$.
Further, we say that $G$ is \emph{$(d,r,\eps)$-rigid}, for some $\eps > 0$, if there are at least $\frac{1-\eps}{r}\binom{n}{2}$ pairs of $(d,r)$-equivalent vertices.
The following theorem, which is the main result of this work, provides a lower bound on the probability that the uniform random graph $G(n,m)$ is $(d,r,\eps)$-rigid for given $d$, $r$, and $\eps$ (that are allowed to depend on $n$ and $m$).

\begin{thm}
  \label{thm:rigidity}
  There exists an absolute constant $C$ such that, for all $\delta,\eps\in(0,1)$ and all nonnegative integers $d, m, n, r$ satisfying $r \ge 2$ and $1 \le m \le (1-\delta)\binom{n}{2}$,
  \[
    \Pr\big(G(n,m) \text{ is not $(d,r,\eps)$-rigid}\big) \le \frac{Cr}{\eps} \cdot \left(\left(\frac{d+1}{\delta}+r\right) \cdot \sqrt{\frac{n}{m}} + \sqrt[4]{\frac{n}{m}}\right).
  \]
\end{thm}

Let us now relate the property of being rigid to the property of admitting a core from the statement of \Cref{thm:main-rigidity-theorem}.
To this end, note first that $(d,r)$-equivalence is an equivalence relation on the vertex set of $G$; we shall call its equivalence classes the \emph{$(d,r)$-components} of~$G$.
In the case when the $r$ largest $(d,r)$-components $C_1, \dotsc, C_r$ of $G$ have strictly more than $n/(r+1)$ vertices each (and thus all remaining components have fewer than $n/(r+1)$ vertices), we call the set $\core_d^r(G) \coloneqq \{C_1,\dotsc,C_r\}$ the \emph{$(d,r)$-core} of $G$.
We will show that, under mild assumptions on the distribution of edges of $G$ (that $G(n,m)$ fails to have with probability much smaller than the upper bound in \Cref{thm:rigidity}), the fact that $G$ is $(d,r,\eps)$-rigid implies that it has a $(d,r)$-core whose each component has at least $n/r-r\eps n$ vertices.  In particular, letting $\Core_d^r(\alpha)$ denote the set of all graphs on $\br{n}$ that have a $(d,r)$-core whose each component has at least $n/r - \alpha n$ vertices, \Cref{thm:rigidity} will imply the following statement.

\begin{cor}
  \label{cor:core}
  There exists an absolute constant $C$ such that, for all $\alpha,\delta\in(0,1)$ and all nonnegative integers $d, m, n, r$ satisfying $r \ge 2$, $\alpha < 1/(r^2+r)$, and $1 \le m \le (1-\delta)\binom{n}{2}$,
  \[
    \Pr\big(G(n,m) \notin \Core_d^r(\alpha)\big) \le \frac{Cr^2}{\alpha} \cdot \left(\left(\frac{d+1}{\delta}+r\right) \cdot \sqrt{\frac{n}{m}} + \sqrt[4]{\frac{n}{m}}\right).
  \]
\end{cor}

\subsection*{Organisation}

The remainder of this section is organised as follows.
First, in \Cref{sec:dynamics-equiv-pairs}, we establish two crucial properties of the dynamics of the sets of $(d,r)$-equivalent pairs of vertices with respect to addition/deletion of edges.
Next, in \Cref{sec:distr-edges-Gnm}, we derive several (standard) estimates regarding the concentration of the number of edges in various induced subgraphs of $\Gnm$ and sizes of the parts in cuts with small deficit.
In \Cref{sc:proof_rigidity}, we use the results established in the previous two subsections to prove~\Cref{thm:rigidity} and in \Cref{sec:derivation-cor-core}, we derive \Cref{cor:core}.
Finally, in \Cref{sec:proof-main-rigidity-thm} we prove \Cref{thm:main-rigidity-theorem}.

\subsection*{Notation}

Given integers $d \ge 0$ and $r \ge 2$, the set of $(d,r)$-equivalent pairs in a graph~$G$ will be henceforth denoted by $\eq_d^r(G)$.  For brevity, we will identify a graph with its set of edges; in particular, we will write $|G|$ to denote the number of edges in a graph~$G$.
Following DeMarco and Kahn~\cite{DeMKah15Tur}, we will use the following notational conventions:
First, given a collection $\Pi$ of pairwise-disjoint sets of vertices in a graph (e.g., a cut), we will denote the set of pairs of vertices contained in a single set of $\Pi$ by $\int(\Pi)$ and the set of remaining pairs of vertices of $\bigcup \Pi$ by $\ext(\Pi)$; this way
\[
  b_r(G) = \max\{|\ext(\Pi) \cap G| : \text{$\Pi$ is an $r$-cut}\}.
\]
Second, given an integer $r \ge 2$, we will write $\crit_r(G)$ to denote the set of edges of $G$ that cross all maximum $r$-cuts (the \emph{$r$-critical} edges of $G$).
Finally, given two families $\cC$ and $\cC'$ of pairwise-disjoint sets of vertices, we will write $\cC \preceq \cC'$ if each element of $\cC$ is contained in some element of $\cC'$.  In other words,
\[
  \cC \preceq \cC' \quad \Longleftrightarrow \quad \forall X \in \cC \; \exists X' \in \cC' \; X \subseteq X'
  \quad \Longleftrightarrow \quad \int(\cC) \subseteq \int(\cC').
\]

\subsection{Dynamics of equivalent pairs}
\label{sec:dynamics-equiv-pairs}

Suppose that $r \ge 2$ and that $G$ is an arbitrary graph on $\br{n}$.
The proof of \Cref{thm:rigidity} will exploit the following crucial property of the dynamics of the sequence $(\eq_d^r(G))_d$ with respect to addition of edges.

\begin{lemma}
  \label{lemma:eq-d-non-edges}
  The following holds for all integers $r \ge 2$ and $d \ge 0$ and every $e \in K_n$:
  If $e \notin G \cup \eq^r_{d+1}(G)$, then $e \notin \eq^r_d(G \cup e)$.
\end{lemma}

Before we prove the lemma, it will be useful to observe the following alternative definitions of the set of $(0,r)$-equivalent pairs. 

\begin{fact}
  \label{fact:eq-analogous-defs}
  The following statements are equivalent for all $e \in K_n \setminus G$:
  \begin{itemize}
  \item
    $e \notin \eq^r_0(G)$;
  \item
    $e \in \ext(\Pi)$ for some maximum $r$-cut $\Pi$ of $G$;
  \item
    $e \in \ext(\Pi)$ for all maximum $r$-cuts $\Pi$ of $G \cup e$  (that is, $e \in \crit_r(G \cup e)$)
  \item
    $b_r(G \cup e) = b_r(G) + 1$.
  \end{itemize}
\end{fact}

\begin{proof}[Proof of~\Cref{lemma:eq-d-non-edges}]
  Suppose that $e \notin G \cup \eq^r_{d+1}(G)$.  Let $\Pi$ be an $r$-cut in $G$ such that $e \in \ext(\Pi)$ and whose deficit $d_e$ in $G$ is the smallest possible.  Note that $d_e\le d+1$ as we assumed that $e$ crosses some cut with deficit at most $d+1$.  If $d_e = 0$, then $\Pi$ is a maximum $r$-cut of $G \cup e$, so in particular $e \notin \eq^r_d(G \cup e)$, due to~\Cref{fact:eq-analogous-defs}.  Otherwise, $b_r(G \cup e) = b_r(G)$ and, consequently, the deficit of $\Pi$ in $G\cup e$ equals
  \[
    b_r(G \cup e) - \left|\ext(\Pi) \cap (G \cup e)\right| = b_r(G) - \left|\ext(\Pi) \cap G\right| - 1 = d_e - 1 \le d.
  \]
  In particular, $\Pi$ is a cut witnessing that $e \notin \eq^r_d(G \cup e)$.
\end{proof}

Both of our applications of \Cref{cor:core} will crucially use the fact that cores are `stable' under small edge perturbations.  Our next lemma, and its corollaries, formalise this notion of stability.

\begin{lemma}
  \label{lemma:eq-d-nested}
  For every integer $d \ge 0$, graph $G \subseteq K_n$, and edge $e \in K_n$, we have
  \[
    \eq_{d+1} (G) \subseteq \eq_d(G \triangle e).
  \]
\end{lemma}
\begin{proof}
  We prove the equivalent statement $K_n \setminus \eq_d(G \triangle e) \subseteq K_n \setminus \eq_{d+1}(G)$.
  Observe first that, for every fixed $r$-cut $\Pi$, the function\footnote{We write $\cP(K_n)$ to denote the family of all subgraphs of $K_n$, i.e., the powerset of $K_n$, which we identify here with its set of edges.}
  \[
    \cP(K_n) \ni H \mapsto \deficit_H(\Pi) = b_r(H) - \big|H \cap \ext(\Pi)\big|
  \]
  is the difference of two nondecreasing, $1$-Lipshitz functions, and thus it is also $1$-Lipshitz.
  In particular, for every graph $G$ and edge $e \in K_n$, we have $\deficit_G(\Pi) \le \deficit_{G \triangle e}(\Pi) + 1$.
  Now, suppose that $f \in K_n \setminus \eq_d(G \triangle e)$, that is, $f \in \ext(\Pi)$ for some $\Pi$ with $\deficit_{G \triangle e}(\Pi) \le d$.  Since $\deficit_G(\Pi) \le \deficit_{G \triangle e}(\Pi) + 1 \le d+1$, we have $f \notin \eq_{d+1}(G)$.
\end{proof}

\begin{cor}
  \label{cor:eq-d-nested}
  For every integer $d \ge 0$ and graphs $G, T \subseteq K_n$, we have
  \[
    \eq_{d+e_T} (G) \subseteq \eq_d(G \triangle T).
  \]
\end{cor}

\begin{cor}
  \label{cor:core-d-nested}
  The following holds for every integer $d \ge 0$ and all graphs $G, T \subseteq K_n$.
  If $G$ has a $(d+e_T, r)$-core, then $G \triangle T$ has a $(d, r)$-core and $\core_{d+e_T}^r(G) \preceq \core_d^r(G \triangle T)$.
\end{cor}

\subsection{Distribution of edges in $\Gnm$}
\label{sec:distr-edges-Gnm}

Our arguments require various estimates of the number of edges in certain subgraphs of $\Gnm$.  Luckily, all these estimates can be easily deduced from the following lower-tail estimate for the number of edges in $\Gnm$ induced by subsets of its vertices.

\begin{lemma}
  \label{lemma:Chernoff_subsets}
  Suppose that $G \sim \Gnm$ for some $m \in \{0, \dotsc, \binom{n}{2}\}$. Then, letting $p = m/\binom{n}{2}$,
  \[
    \Pr\left(\exists U\,\,|G[U]| < \binom{|U|}{2}p-2|U|\sqrt{np}\right) \le e^{-n}.
  \]
\end{lemma}

\begin{proof}
  We may clearly assume that $m \ge 1$.
  Fix a nonempty $U\subseteq\br{n}$.  Standard estimates on lower tail probabilities of the binomial distribution yield\footnote{See \cite[Section~6]{Hoe63}, which argues that the hypergeometric distribution is at least as concentrated as the binomial distribution with the same parameters.}
  \[
    \Pr\left(|G[U]| < \binom{|U|}{2}p-2|U|\sqrt{np}\right)\leq\exp\left(-\frac{4|U|^2np}{p|U|^2}\right)=e^{-4n}.
  \]
  The union bound over all $U$ finishes the proof.
\end{proof}

In fact, we will only use the following immediate corollary of \cref{lemma:Chernoff_subsets}.

\begin{cor}
  \label{cor:equiv_classes_edges_count}
  Suppose that $G \sim \Gnm$ for some $m \in \{0, \dotsc, \binom{n}{2}\}$.
  With probability at least $1 - e^{-n}$, for every partition $\Pi$ of $\br{n}$,
  \[
    |\int(\Pi) \cap G| \ge |\int(\Pi)| \cdot p - 2n\sqrt{np}.
  \]
\end{cor}

One important consequence of \Cref{cor:equiv_classes_edges_count} is that every cut of $\Gnm$ with small deficit must be balanced.

\begin{lemma}
  \label{lemma:max_cut_balanced}
  Suppose that $G \sim \Gnm$ for some $m \in \{0, \dotsc, \binom{n}{2}\}$.
  With probability at least $1-e^{-n}$, each part of every $r$-cut of $G$ with deficit at most $d$ has at most $n/r+4\max\{(d/m)^{1/2}, (n/m)^{1/4}\} \cdot n$ vertices.
\end{lemma}

The proof of~\Cref{lemma:max_cut_balanced} uses the following straightforward estimate on the number of edges that cross a non-balanced cut.

\begin{fact}
  \label{fact:Pi-unbalanced}
  If $\Pi = (A_1, \dotsc, A_r)$ is an $r$-cut in $K_n$ with $\max_i |A_i| = n/r+\eps n$ for some $\eps \ge 0$, then
  \[
    |\ext(\Pi)| \le \left(1 - \frac{1}{r} - \frac{r\eps^2}{r-1}\right) \cdot \binom{n}{2}+ \frac{n}{2}.
  \]
\end{fact}
\begin{proof}
  We may clearly assume that $|A_1| = \max_i |A_i| = n/r+\eps n$.  By convexity of $x \mapsto \binom{x}{2}:=\frac{x(x-1)}{2}$,
  \[
    |\ext(\Pi)| = \binom{n}{2} - \sum_{i=1}^r \binom{|A_i|}{2} \le \binom{n}{2} - \binom{n/r+\eps n}{2} - (r-1) \cdot \binom{n/r-\eps n/(r-1)}{2}.
  \]
  Further, as for all nonnegative $\lambda_1, \dotsc, \lambda_r$ with $\lambda_1 + \dotsb + \lambda_r = 1$, we have
  \[
    \sum_{i=1}^r \left(\lambda_i^2 \binom{n}{2} - \binom{\lambda_in}{2}\right) = \sum_{i=1}^r \frac{(1-\lambda_i)\lambda_in}{2} \le \frac{n}{2},
  \]
  we may conclude that
  \[
    \begin{split}
      |\ext(\Pi)| & \le \left(1 - \left(\frac{1}{r}+\eps\right)^2 - (r-1) \left(\frac{1}{r} - \frac{\eps}{r-1}\right)^2\right) \cdot \binom{n}{2}+\frac{n}{2}\\
      & = \left(1 - \frac{1}{r} - \frac{r\eps^2}{r-1}\right) \cdot \binom{n}{2}+\frac{n}{2},
    \end{split}
  \]
  as claimed.
\end{proof}

\begin{proof}[Proof of~\Cref{lemma:max_cut_balanced}]
  Assume the assertion of \Cref{cor:equiv_classes_edges_count}, which holds with probability at least $1-e^{-n}$.
  Let $\Pi = (A_1,\dotsc,A_r)$ be an $r$-cut with deficit at most $d$ and let $\eps$ be the number satisfying $\max_i |A_i| = n/r+\eps n$.
  On the one hand, since every graph $G$ with $m$ edges satisfies $b_r(G) \ge \frac{r-1}{r} \cdot m$ (to see this, consider a uniformly random $r$-cut), the size of $\Pi$ in $G$ is at least $\frac{r-1}{r} \cdot m - d$.
  On the other hand, by the assumed assertion of \cref{cor:equiv_classes_edges_count} and \cref{fact:Pi-unbalanced}, letting $p \coloneqq m / \binom{n}{2}$,
  \[
    |\ext(\Pi) \cap G| \le |\ext(\Pi)| \cdot p + 2n\sqrt{np} \le \left(\frac{r-1}{r} - \frac{r\eps^2}{r-1}\right) \cdot m + \frac{np}{2} + 2n\sqrt{np}.
  \]
  This yields
  \[
    \frac{r\eps^2m}{r-1} \le d + \frac{np}{2} + 2n\sqrt{np} \le d + \frac{5n\sqrt{np}}{2} \le d + 4\sqrt{nm},
  \]
  which means that
  \[
    \max_i |A_i| - \frac{n}{r} = \eps n \le 2\max\left\{\sqrt{\frac{d}{m}}, \left(\frac{16n}{m}\right)^{1/4}\right\} \cdot n,
  \]
  as desired.
\end{proof}

\subsection{Proof of \Cref{thm:rigidity}}
\label{sc:proof_rigidity}

It will be convenient to define, for a graph $G$ on $\br{n}$,
\[
  \eq^r_{-1}(G) \coloneqq \int(\Pi_G),
\]
where $\Pi_G$ is some (canonically chosen) maximum $r$-cut in $G$.  By convexity of $x \mapsto \binom{x}{2}$,
\begin{equation}
  \label{eq:eq_-1^r}
  \left| \eq_{-1}^r(G) \right| \ge r \cdot \binom{n/r}{2} = \frac{n^2}{2r} - \frac{n}{2}.
\end{equation}
The following lemma lies at the heart of the proof of \Cref{thm:rigidity}.
Throughout this section, we denote $N \coloneqq \binom{n}{2}$.

\begin{lemma}
  \label{lemma:eq-d-d-plus-1}
  Let $G_1 \sim \Gnm$ and $G_2 \sim \Gnmp$ for some $m \in \{0, \dotsc, N-1\}$.
  For every $r \ge 2$ and all $d \ge -1$,
  \[
    \Ex\left|\eq^r_{d+1}(G_1) \setminus G_1\right| \ge  \frac{N-m}{m+1} \cdot \Ex\left|\eq^r_d(G_2) \cap G_2\right|.
  \]
\end{lemma}
\begin{proof}
  We may clearly couple $G_1$ and $G_2$ so that $G_2 = G_1 \cup e$ and $e$ is a uniformly random edge of each of $K_n \setminus G_1$ and $G_2$.  Define, for every $d \ge -1$,
  \[
    q_d \coloneqq \Pr\big(e \notin \eq^r_d(G_1)\big)
  \]
  and observe that
  \begin{equation}
    \label{eq:exp_setminus}
    q_d \cdot (N-m)  = \Ex\left|(K_n \setminus G_1) \setminus \eq^r_d(G_1)\right|
    = N - m - \Ex\left|\eq^r_d(G_1)\setminus G_1\right|.
  \end{equation}
  On the other hand, \Cref{lemma:eq-d-non-edges} implies that, for every $d \ge 0$,
  \begin{equation}
    \label{eq:exp_intersection}
    q_{d+1} \cdot (m+1) \le \Pr\big(e \notin \eq^r_d(G_2)\big) \cdot (m+1) = m+1 - \Ex\left|\eq^r_d(G_2) \cap G_2\right|.
  \end{equation}
  This inequality remains valid also when $d=-1$ since, by \Cref{fact:eq-analogous-defs},
  \[
    q_0 \cdot (m+1) = \Ex\left|\crit_r(G_2)\right| \le \Ex[b_r(G_2)]=
    m+1-\Ex\left|\eq^r_{-1}(G_2) \cap G_2\right|.
  \]
  The claimed inequality follows by combining~\eqref{eq:exp_setminus}, with $d$ replaced by $d+1$, and~\eqref{eq:exp_intersection}.  
\end{proof}

We now use \Cref{cor:equiv_classes_edges_count} to derive the following cleaner statement.

\begin{cor}
  \label{cor:eq-d-d-plus-1}
  Let $G_1 \sim \Gnm$ and $G_2 \sim \Gnmp$ for some $m \in \{0, \dotsc, N-1\}$.
  For every $r \ge 2$, all $d \ge -1$, and all sufficiently large $n$,
  \[
    \Ex\left|\eq_{d+1}^r(G_1)\right| \ge \Ex\left|\eq_d^r(G_2)\right| - \frac{2n^{9/2}}{\sqrt{m+1} \cdot (N-m)}.
  \]
\end{cor}
\begin{proof}
  Since, for all $d$, $r$, and $G$, the graph $\eq_d^r(G)$ is a union of cliques, \Cref{cor:equiv_classes_edges_count} gives, letting $p_1 \coloneqq m/N$ and $p_2 \coloneqq (m+1)/N$,
  \[
    \Ex\left| \eq^r_d(G_2) \cap G_2 \right| 
    \ge \Ex\left|\eq^r_d(G_2)\right| \cdot p_2-2n\sqrt{np_2}-n^2e^{-n} \ge \Ex\left|\eq^r_d(G_2)\right| \cdot p_2-3n\sqrt{np_2},
  \]
  where the last inequality holds for all sufficiently large $n$, and, similarly,
  \[
    \begin{split}
      \Ex\left| \eq^r_{d+1}(G_1)\setminus G_1\right| & = \Ex\left| \eq^r_{d+1}(G_1) \right| - \Ex\left|\eq_{d+1}^r(G_1) \cap G_1\right| \\
      & \le \Ex\left|\eq^r_{d+1}(G_1)\right|\cdot(1-p_1)+3n\sqrt{np_1}.
  \end{split}
  \]
  Consequently, \Cref{lemma:eq-d-d-plus-1} yields
  \[
    \begin{split}
      \Ex\left|\eq^r_{d+1}(G_1)\right|
      &\ge \frac{1}{1-p_1} \cdot \left(\Ex\left| \eq^r_{d+1}(G_1) \setminus G_1\right|-3n\sqrt{np_1}\right)\notag \\
      &\ge \frac{1}{1-p_1} \cdot \frac{N-m}{m+1} \cdot \Ex\left|\eq^r_d(G_2) \cap G_2 \right|-\frac{3n\sqrt{np_1}}{1-p_1}\notag \\
      &\ge \frac{p_2}{1-p_1} \cdot \frac{N-m}{m+1} \cdot \left(\Ex\left|\eq^r_d(G_2)\right|-\frac{3n\sqrt{np_2}}{p_2}\right)-\frac{3n\sqrt{np_1}}{1-p_1} \\
      & = \Ex\left|\eq^r_d(G_2)\right| - 3n^{3/2} \cdot \left(\frac1{\sqrt{p_2}} + \frac{\sqrt{p_1}}{1-p_1}\right).
    \end{split}
  \]
  The claimed inequality follows as
  \[
    \frac{1}{\sqrt{p_2}} + \frac{\sqrt{p_1}}{1-p_1} \le \frac{1}{\sqrt{p_2}} + \frac{\sqrt{p_2}}{1-p_1} = \frac{1-p_1+p_2}{\sqrt{p_2} \cdot (1-p_1)} = \frac{\sqrt{N} \cdot (N+1)}{\sqrt{m+1} \cdot (N-m)}.\qedhere
  \]
\end{proof}

We conclude this subsection with the proof of \Cref{thm:rigidity}, which now amounts to little more than summing the inequalities from the assertion of \Cref{cor:eq-d-d-plus-1} over a~range of $d$.

\begin{proof}[Proof of \Cref{thm:rigidity}]
  Let $\delta, \eps \in (0,1)$ be arbitrary and suppose that nonnegative integers $d, m, n, r$ satisfy the assumptions of the theorem.  Note that we may assume that $d \le \sqrt{m/n}$ and that $n$ is sufficiently large, as otherwise the asserted upper bound on the probability is larger than one.
  Let $G \sim \Gnm$ and let $G' \sim \Gnmdp$.
  It follows from \Cref{cor:eq-d-d-plus-1} that
  \[
    \begin{split}
      \Ex\left|\eq_d^r(G)\right| - \Ex\left|\eq_{-1}^r(G')\right|
      & \ge - \sum_{i=0}^{d} \frac{2n^{9/2} }{\sqrt{m+i+1} \cdot (N-m-i)} \\
      & \ge - \frac{2(d+1)n^{9/2}}{\sqrt{m} \cdot (\delta N-d)}
      \ge - \frac{5(d+1)n^{5/2}}{\delta\sqrt{m}}.
    \end{split}
  \]
  In particular, due to~\eqref{eq:eq_-1^r},
  \begin{equation}
    \label{eq:Ex-eq-d-r-lower}
    \Ex\left|\eq^r_d(G)\right| \ge \frac{N}{r} - \frac{n}{2} - \frac{5(d+1)n^{5/2}}{\delta \sqrt{m}}.
  \end{equation}
  On the other hand, since $(d/m)^{1/2} \le (mn)^{-1/4} \le (n/m)^{1/4}$, \Cref{lemma:max_cut_balanced} implies that, with probability at least $1-e^{-n}$, each part of every $r$-cut with deficit at most $d$ has at most $n/r+4(n/m)^{1/4} \cdot n$ vertices.  Thus, letting $a \coloneqq 4(n/m)^{1/4} \cdot n$, we have
  \[
    \begin{split}
      \left|\eq^r_d(G)\right|
      & \le r \cdot \binom{n/r+a}{2} \le r \cdot \binom{n/r}{2} + r \cdot \left(\frac{n}{r}+\frac{a}{2}\right)a \le \frac{N}{r} + na + \frac{ra^2}{2} \\
      & = \frac{N}{r} + 4(n/m)^{1/4} \cdot n^2 + 8r(n/m)^{1/2} \cdot n^2.
    \end{split}
  \]
  with probability at least $1-e^{-n}$.
  Finally, let
  \[
    P \coloneqq \Pr\big(\text{$G$ is not $(d,r,\eps)$-rigid}\big) = \Pr\big(\eq_d^r(G) \le (1-\eps)N/r\big).
  \]
  We may conclude that, for some absolute constant $C$,
  \begin{equation}
    \label{eq:Ex-eq-d-r-upper}
    \Ex\left|\eq_d^r(G)\right| \le P \cdot \frac{(1-\eps)N}{r} + (1-P) \cdot \left(\frac{1}{r} + C\sqrt[4]{\frac{n}{m}} + Cr\sqrt{\frac{n}{m}}\right) \cdot N + e^{-n} \cdot N.
  \end{equation}
  Combining~\eqref{eq:Ex-eq-d-r-lower} and~\eqref{eq:Ex-eq-d-r-upper}, we obtain
  \[
    P \le \frac{r}{\eps} \cdot \left(\frac{5(d+1)n^{5/2}}{\delta \sqrt{m} N} + C\sqrt[4]{\frac{n}{m}} + Cr\sqrt{\frac{n}{m}}\right),
  \]
  provided that $C$ is sufficiently large, as claimed.
\end{proof}

\subsection{Derivation of \Cref{cor:core}}
\label{sec:derivation-cor-core}

The following key lemma, which is a variant of \cite[Proposition~10.1]{DeMKah15Tur} (see also \cite[Proposition~7.2]{HosSam}), implies that every graph that is sufficiently rigid must contain a core, provided that each of its cuts with small deficit is balanced.

\begin{lemma}
  \label{lemma:core}
  Suppose that $r \ge 2$ and $\alpha \in (0, 1/(r^2+r))$ and let $d \ge 0$ be an arbitrary integer.
  If an $n$-vertex graph $G$ is $(d, r, \alpha/r)$-rigid and each part of every $r$-cut of $G$ with deficit at most $d$ has size at most $(1+\alpha)n/r$, then the $r$ largest $(d,r)$-components of $G$ have size at least $n/r-\alpha n$ each.
\end{lemma}
\begin{proof}
  Suppose that each part of every $r$-cut of $G$ with deficit at most $d$ has size at most $(1+\alpha)n/r$.  If only at most $r-1$ largest $(d,r)$-components of $G$ had size at least $n/r-\alpha n$, then it would follow from convexity of $x \mapsto \binom{x}{2}$ that
  \[
    \begin{split}
      |\eq_d^r(G)| & \le (r-1) \binom{(1+\alpha)n/r}{2} + \binom{n/r-\alpha n}{2} + \binom{\alpha n/r}{2} \\
      & \le \left((r-1)\big((1+\alpha)/r\big)^2+(1/r-\alpha)^2+(\alpha/r)^2\right)\binom{n}{2} \\
      & = \frac{1-2\alpha/r+(r+1)\alpha^2}{r} \binom{n}{2} < \frac{1-\alpha/r}{r} \binom{n}{2},
    \end{split}
  \]
  which means that $G$ would not be $(d, r, \alpha/r)$-rigid.
\end{proof}

\begin{proof}[Proof of~\cref{cor:core}]
  Let $\alpha, \delta \in (0,1)$ be arbitrary and suppose that nonnegative integers $d, m, n, r$ satisfy the assumptions of the corollary.  Note that we may assume that $d \le \sqrt{m/n}$ and $4r^2/\alpha < \sqrt[4]{m/n}$, as otherwise the asserted upper bound on the probability is larger than one.  Since these assumptions imply that
  \[
    \frac{\alpha n}{r} > 4 \cdot \sqrt[4]{\frac{n}{m}} \cdot n = 4 \cdot \max\left\{\sqrt{\frac{d}{m}}, \sqrt[4]{\frac{n}{m}}\right\} \cdot n,
  \]
  we may invoke \cref{lemma:max_cut_balanced} to conclude that, with probability $1 - e^{-n}$, each part of every $r$-cut in $\Gnm$ with deficit at most $d$ has no more than $(1+\alpha) n/r$ vertices.  Consequently, by~\cref{lemma:core},
  \[
    \Pr\big(\Gnm \notin \Core_d^r(\alpha)\big) \le \Pr\big(\Gnm \text{ is not $(d,r,\alpha/r)$-rigid}\big) + e^{-n}.
  \]
  The assertion of the corollary now follows immediately from \Cref{thm:rigidity}.
\end{proof}

\subsection{Proof of \Cref{thm:main-rigidity-theorem}}
\label{sec:proof-main-rigidity-thm}

Since our assumption that $pn \gg 1$ implies that whp $\Gnp$ has $(1+o(1))\binom{n}{2}p$ edges, the first assertion of the theorem is a straightforward consequence of \Cref{cor:core}.
It thus remains to prove the second assertion.
The following observation is at the heart of the matter.
Given a graph $G$, a vertex $v \in V(G)$, and a set $W \subseteq V(G)$, we write $\deg_G(v,W)$ to denote the number of neighbours of $v$ in $W$.

\begin{lemma}
  \label{lemma:non-rigidity}
  Suppose that some largest $r$-cut $\Pi = \{V_1, \dotsc, V_r\}$ in a graph $G$ and a vertex $v \in V_1$ satisfy
  \[
    \deg_G(v, V_1) \ge \frac{\deg_Gv}{r} - \frac{r-1}{r} \cdot d
  \]
  for some nonnegative integer $d$.  Then $v$ is not $(d,r)$-equivalent to any other vertex of $G$.
\end{lemma}
\begin{proof}
  For each $i \in \{2, \dotsc, r\}$, denote by $\Pi_i$ the $r$-cut obtained from $\Pi$ by moving the vertex $v$ from $V_1$ to $V_i$ and observe that
  \[
    \begin{split}
      |\ext(\Pi_i) \cap G|
      & = |\ext(\Pi) \cap G| + \deg_G(v, V_1) - \deg_G(v, V_i) \\
      & = b_r(G) + \deg_G(v, V_1) - \deg_G(v, V_i).
    \end{split}
  \]
  Consequently, writing $\deficit_G(\Pi_i) \coloneqq b_r(G) - |\ext(\Pi_i) \cap G|$ for the deficit of $\Pi_i$ in $G$, we have
  \[
    \sum_{i=2}^r \deficit_G(\Pi_i) = \sum_{i=2}^r \left(\deg_G(v, V_i) - \deg_G(v, V_1)\right) = \deg_Gv - r\deg_G(v,V_1) \le (r-1)d.
  \]
  In particular, there must be some $i \in \{2, \dotsc, r\}$ for which $\deficit_G(\Pi_i) \le d$.  The assertion now follows as no vertex $u \neq v$ shares with $v$ the same part of both $\Pi$ and $\Pi_i$.
\end{proof}

Suppose now that $G \sim \Gnp$.
Let $\Pi$ be an arbitrary largest $r$-cut in $G$ and define, for each nonnegative integer $d$,
\[
  X_d \coloneqq \left\{v \in V(G) : v \in V \in \Pi \text{ and } \deg_G(v, V) < \frac{\deg_Gv}{r} - \frac{r-1}{r} \cdot d\right\}.
\]
Since \Cref{lemma:non-rigidity} implies that $\eq_d^r(G) \subseteq \binom{X_d}{2}$, it suffices to show that whp $|X_d| < \alpha n$ when $d \ge C\sqrt{pn}$ for large enough $C$.

The assumption that $\Pi$ is a largest $r$-cut in $G$ implies that $\deg_G(v,V) \le (\deg_Gv)/ r$ for each $v \in V \in \Pi$ and therefore
\[
  \frac{r-1}{r} \cdot \left(2|G| + |X_d| \cdot d\right) \le \sum_{v \in V \in \Pi} \left(\deg_Gv - \deg_G(v, V)\right) = 2\left|\ext(\Pi) \cap G\right|.
\]
This yields
\[
  |X_d| \le \frac{2r \cdot |\ext(\Pi) \cap G| - 2(r-1) \cdot |G|}{(r-1) \cdot d} = \frac{2|G| - 2r \cdot |\int(\Pi) \cap G|}{(r-1) \cdot d}.
\]
In particular, it follows from \Cref{cor:equiv_classes_edges_count} that, with probability at least $1 - e^{-n}$,
\[
  |X_d| \le \frac{2|G| - 2r \cdot |\int(\Pi)| \cdot |G| / \binom{n}{2} + 4r \cdot n\sqrt{n|G|/\binom{n}{2}}}{(r-1) \cdot d}.
\]
Finally, since $r \cdot |\int(\Pi)| \ge r^2 \cdot \binom{n/r}{2} \ge \binom{n}{2} - \frac{rn}{2}$, by convexity of $x \mapsto \binom{x}{2}$, and since whp we have $|G| \le 2\binom{n}{2}p$ and
\[
  |X_d| \le \frac{2rnp + 6rn\sqrt{np}}{(r-1) \cdot d} < \frac{7rn\sqrt{np}}{(r-1) \cdot d}.
\]
In particular, if $d \ge 14\sqrt{np}/\alpha$, then $|X_d| < \alpha n$, as
desired.

\subsection{Non-uniqueness of maximum cuts and vertices outside of the core}
\label{subsection:complement-of-the-core}

In this subsection, we present supporting evidence for our suspicion that whp $\Gnm$ does not have a unique maximum cut for all $m$ satisfying $n \ll m \le N - \Omega(N)$.\footnote{This is vacuously true when $m \le n \log n / 2 - \omega(n)$, since then whp $\Gnm$ has isolated vertices.}
By \cite[Lemma~14]{C-OMS06}, with probability very close to one, the maximum size of an $r$-cut in $\Gnm$ is $\frac{r-1}{r} \cdot m + \Theta(\sqrt{mn})$.
It is reasonable to believe that the expectation of this random variable grows `smoothly' with $m$ in the sense that $\Ex[b_r(\Gnmp) - b_r(\Gnm)] = \frac{r-1}{r} + \Theta(\sqrt{n/m})$.
If this estimate held conditionally on $\Gnm$, under the natural coupling for which $\Gnm \subseteq \Gnmp$, that is, if
\begin{equation}
  \label{eq:partial-br}
  \partial b_r \coloneqq \Ex[b_r(\Gnmp) \mid \Gnm] - b_r(\Gnm) \ge \frac{r-1}{r} + \Theta(\sqrt{n/m}),
\end{equation}
then $\Gnm$ would not have a unique maximum $r$-cut.
Indeed, it follows from \cref{fact:eq-analogous-defs} that, if $\Pi$ is the unique maximum $r$-cut of $\Gnm$, then
\[
  \partial b_r = \frac{|\ext(\Pi) \setminus \Gnm|}{\binom{n}{2} - m} \le \frac{r-1}{r}.
\]
Unfortunately, we do not even know how to prove that \eqref{eq:partial-br} holds with probability $\Omega(1)$.
Instead, we show that the expected number of vertices that lie outside of the core is large in either $\Gnm$ or in $\Gnmp$, and thus also in $\Gnp$.

Given a graph $G$ and an integer $r \ge 2$, denote by $X_r(G)$ the set of vertices not belonging to the $(0,r)$-core of $G$ (so that $X_r(G) = V(G)$ if $G$ does not have a $(0,r)$-core).
Further, set $x_r(G) \coloneqq |X_r(G)|$.
By the definition of the $(0,r)$-core, $x_r(G) = 0$ if and only if $G$ has a unique largest $r$-cut whose each part has more than $|V(G)|/(r+1)$ vertices.\footnote{\Cref{lemma:max_cut_balanced} implies that, for all $m \gg n$, with probability at least $1 - e^{-n}$, each part of every maximum $r$-cut of $\Gnm$ has $n/r \pm o(n)$ vertices.} We show that our upper bound on $\Ex[x_r(\Gnp)]$ provided by~\eqref{eq:Ex-eq-d-r-lower} is tight up to a constant factor. 

\begin{prop}\label{prop:vertices-outside-the-core}
  Let $r \ge 2$ be an integer and let $\delta$ be a positive constant.
  If $1/n \ll p \le 1 - \delta$, then
	$$\Ex[x_r(\Gnp)] = \Omega\left(\sqrt{n/p}\right).$$
\end{prop}

We postpone the proof of Proposition~\ref{prop:vertices-outside-the-core} to Appendix in order to avoid overloading this section with extra technical details since its proof method resembles the proof of Corollary~\ref{cor:eq-d-d-plus-1}.

\section{Neighbourhood in maximum cuts}
\label{sec:neighbourhood}

In this section, we prove \Cref{thm:neighbourhood}.
It will be convenient to first introduce some notation.
Given a graph $G$, a vertex $v \in V(G)$, and a set of vertices $S \subseteq V(G)$, we denote by $N_G(v)$ and $N_G(v; S)$ the sets of neighbours of $v$ in $V(G)$ and in $S$, respectively.
Further, given a set $W \subseteq V(G)$, we denote by $N_G(W)$ and by $N_G(W;S)$ the set of common neighbours of $W$ in $V(G)$ and in $S$, respectively; note that usually $N_G(W)$ denotes the union of $N_G(v)$ over $v\in W$, rather than their intersection.
That is,
\[
  N_G(W) \coloneqq \bigcap_{v \in W} N_G(v)
  \qquad\text{and}\qquad
  N_G(W; S) \coloneqq \bigcap_{v \in W} N_G(v; S).
\]
Finally, we will denote by $\Hyp(N,s,t)$ the size of the intersection of a random $t$-element subset of $\br{N}$ with the set $\br{s}$ (this is a hypergeometric random variable).

Let us fix integers $k \ge 1$, $r \ge 2$ and a small $\eps > 0$.
Further, assume that $p$ satisfies
\[
  C \left(\frac{\log n}{n}\right)^{1/k} \le p \le 1 -\eps
  \qquad
  \text{and}
  \qquad
  p \gg \frac{(\log n)^2}{n}
\]
for a sufficiently large $C>0$.
Our aim is to prove that whp every maximum $r$-cut $(V_1, \dotsc, V_r)$ of $\bG \sim \Gnp$ and every $k$-element $W \subseteq V(G)$ satisfy
\[
  |N_{\bG}(W;V_i)| = \left(\frac{1}{r} \pm \eps\right) np^k
  \quad
  \text{for all $i \in \br{r}$}.
\]
Denote by $\cB$ the event that there exists a set $W$ of $k$ vertices and a maximum $r$-cut $(V_1, \dots, V_r)$ of $\bG$ such that, for some $i\in\br{r}$,
\[
  |N_{\bG}(W; V_i)| \le \frac{1 - \eps}{r} \cdot |N_{\bG}(W)|.
\]
Moreover, for a $k$-set $W\subseteq V(\bG)$, denote by $\cY_W$ the event that $|N_{\bG}(W)| = (1\pm\eps^2)np^k$ and set $\cY \coloneqq \bigcap_W\cY_W$.
Since on the event $\cB^c \cap \cY$, we have
\begin{multline*}
  \left(\frac{1}{r}-\eps \right)np^k
  \le \frac{1 - \eps}{r} \cdot \left(1-\eps^2\right)np^k
  \le |N_{\bG}(W;V_i)|
  \le |N_{\bG}(W)| - \sum_{j\neq i}|N_{\bG}(W;V_j)| \\
  \le \left( 1 + \eps^2 \right)np^k - (r-1) \cdot \frac{1 - \eps}{r}np^k
  \le \left( \frac{1}{r}+\eps \right)np^k,
\end{multline*}
in order to complete the proof of \Cref{thm:neighbourhood}, it is sufficient to prove that
\begin{equation}
  \Pr(\cY^c) + \Pr(\cB)=o(1).
  \label{eq:neighb_proof_1}
\end{equation} 

By the Chernoff bound and our assumption on $p$ and the choice of $C$, for every $k$-element $W
\subseteq V(G)$,
\[
  \Pr\left( |N_{\bG}(W)| \neq (1\pm\eps^2)np^k \right) \le \exp\left(
    -\eps^4np^k/5 \right) \le n^{-2k};
\]
the estimate $\Pr(\cY^c) = o(1)$ follows from a straightforward union bound over all $W$.

We now turn to bounding $\Pr(\cB)$.
Let $\alpha$ be a small positive constant.
Recall that $\bG \in \Core_d^r(\alpha)$ means that $\bG$ has a $(d, r)$-core whose each component has at least $n/r-\alpha n$ vertices and that $\core_d^r(\bG)$ denotes the set of the $r$ components of the $(d, r)$-core of
$\bG$.
(If $\bG \notin \Core_d^r(\alpha)$, then we set $\core_d^r(\bG) = \emptyset$.)
For a $k$-set $W \subseteq V(\bG)$, denote by $\cA_W$ the event that
\begin{equation}
  \label{eq:A_W-condition}
  \left|N_\bG(W; S)\right| \le \frac{1 - \eps}{r} \cdot \left|N_\bG(W)\right|
  \quad
  \text{for some $S \in \core_0^r(\bG)$}
\end{equation}
and let $\cA \coloneqq \bigcup_W \cA_W$.
Set $t \coloneqq \lceil\sqrt{C} \log n\rceil$ and let $d \coloneqq kt$.
The key observation is that, if $\bG$ has a $(0,r)$-core and $\cB$ holds, then $\cA$ holds as well.
Therefore,
\[
  \Pr(\cB) \le \mathbb(\bG \notin \Core_{2d}^r(\alpha)) + \Pr(\bG \in \Core_{2d}^r(\alpha) \wedge \cA).
\]
The first summand above is $o(1)$, by \Cref{cor:core}, whereas
\begin{equation}
  \label{eq:neighb_proof_3}
  \Pr(\bG \in \Core_{2d}^r(\alpha) \wedge \cA)
  \le \Pr(\cY^c) + \sum_{W} \Pr(\bG \in \Core_{2d}^r(\alpha) \wedge \cA_W \mid \cY_W).
\end{equation}
Since we have already shown that $\Pr\left(\cY^c\right)=o(1)$, it suffices to prove that each term in the sum is much smaller than $n^{-k}$.

To this end, fix a $k$-element set $W$ and, assuming that $\cY_W$ holds, let $\bR$ be a uniformly chosen $t$-element subset of $N_{\bG}(W)$ (clearly, $\cY_W$ implies that $|N_{\bG}(W)| \ge t$).
Denote by $\cA_{W,t}$ the event that
\begin{equation}
  \label{eq:neighb_proof_2}
  |\bR \cap S| \le \frac{1 - \eps/2}{r} \cdot t
  \quad
  \text{for some $S \in \core_0^r(\bG)$}
\end{equation}
and observe that
\begin{equation}
  \label{eq:A_W-and-A_W^t}
  \begin{split}
    \Pr(\bG \in \Core_{2d}^r(\alpha) \wedge \cA_{W,t} \mid \cY_W)
    & \ge \Pr(\bG \in \Core_{2d}^r(\alpha) \wedge \cA_{W} \mid \cY_W) \\
    & \qquad \cdot \Pr(\cA_{W,t} \mid \bG \in \Core_{2d}^r(\alpha) \wedge \cA_{W} \wedge \cY_W).
  \end{split}
\end{equation}
If $\cY_W$ holds (so that $\bR$ is defined) and $\bG$ has a $(0,r)$-core, then, for every $S \in \core_0^r(\bG)$, the cardinality $|\bR \cap S|$ is distributed like $\Hyp(|N_{\bG}(W)|, |N_{\bG}(W;S)|, t)$.
In particular, it follows from standard tail bounds for the hypergeometric distribution~\cite[Section~6]{Hoe63} that whp (in the choice of $\bR$) \eqref{eq:A_W-condition} implies \eqref{eq:neighb_proof_2}.
Thus, the second factor in the right-hand side of~\eqref{eq:A_W-and-A_W^t} approaches $1$, implying
\begin{equation}
  \label{eq:upper-bound-for-A_W}
  \Pr(\bG \in \Core_{2d}^r(\alpha) \wedge \cA_{W} \mid \cY_W) \le (1+o(1)) \cdot \Pr(\bG \in \Core_{2d}^r(\alpha) \wedge \cA_{W,t} \mid \cY_W).
\end{equation}
Now, if the distribution of $\bR$, conditioned on $\core_0^r(\bG)$, was the uniform distribution on $t$-element subsets of $V\setminus W$, then, for every $S \in \core_0^r(\bG)$, the cardinality $|\bR \cap S|$ would again have hypergeometric distribution, implying that $\cA_{W,t} $ has exponentially small probability.
We will show something marginally weaker:
The distribution of $\bR$ is approximately uniform after we condition on the graph $\bGs$ that is obtained from $\bG$
by resampling all edges between $W$ and $\bR$.
Since $\bGs$ satisfies $\core_d^r(\bGs) \preceq \core_0^r(\bG)$, by \Cref{cor:core-d-nested},
this will allow us to prove the following estimate.

\begin{claim}
  \label{claim:upper-bound-for-A_W-t}
  There exists a $c = c(r,\alpha,\eps) > 0$ such that,
  for every $k$-element set $W$,
  \[
    \Pr\left(\bG \in \Core_{2d}^{r}(\alpha) \wedge \cA_{W,t} \mid \cY_W\right) \le e^{-ct}.
  \]
\end{claim}

Given \Cref{claim:upper-bound-for-A_W-t}, it is now easy to finish the proof of \Cref{thm:neighbourhood}.
Indeed, by \eqref{eq:upper-bound-for-A_W},
\[
  \Pr(\bG \in \Core_{2d}^r(\alpha) \wedge \cA_{W} \mid \cY_W) \le (1+o(1)) \cdot e^{-ct} \ll n^{-k}
\]
due to our choice of $t$.
Substituting this bound into~\eqref{eq:neighb_proof_3} completes the proof.

\begin{proof}[Proof of \Cref{claim:upper-bound-for-A_W-t}]
  Fix some $k$-element set of vertices $W$, assume that $\cY_W$ holds, and let $\bGs$ be the graph obtained from $\bG$ by independently resampling each edge between $W$ and $\bR$ with probability $p$.
  We shall first show that, conditioned on $\bGs$, $\bR$ is a~nearly-uniformly random $k$-element subset of $\br{n} \setminus W$.

  To see this, fix some $t$-element $R \subseteq \br{n} \setminus W$ and a graph $G^*$ such that $\bGs = G^*$ with nonzero probability.
  Writing $G^*[W,R]$ for the induced bipartite subgraph of $G^*$ with parts $W$ and $R$ and $[W,R]$ for the complete bipartite graph with parts $W$ and $R$, we have
  \[
    \begin{split}
      \Pr(\bR = R \wedge \bGs = G^*)
      & = \Pr(\bG = G^* \cup [W,R]) \cdot \Pr(\bR = R \mid \bG = G^* \cup [W,R])\\
      & \qquad \cdot \Pr(\bGs[W, R] = G^*[W, R] \mid \bG = G^* \cup [W,R] \wedge \bR = R) \\
      & =  p^{|G^* \cup [W, R]|} (1-p)^{\binom{n}{2} - |G^* \cup [W,R]|} \cdot \binom{|N_{G^* \cup [W,R]}(W)|}{t}^{-1} \\
      & \qquad \cdot p^{|G^*[W, R]|} (1-p)^{|W||R| - |G^*[W, R]|} \\
      & = p^{|G^*| + kt} (1-p)^{\binom{n}{2} - |G^*|} \cdot
        \binom{|N_{G^*}(W)| - |N_{G^*}(W) \cap R| + t}{t}^{-1}.
    \end{split}
  \]
  Since $|N_\bGs(W)| \le |N_\bG(W)| \le (1+\eps^2)np^k$ and
  \[
    |N_\bGs(W)| - t \ge |N_\bG(W)| - 2t \ge (1-\eps^2)np^k - 2t \ge (1-2\eps^2)np^k,
  \]
  we have, for any two $t$-element subsets $R, R' \subseteq \br{n} \setminus W$,
  \[
    \frac{\Pr(\bR = R \mid \bGs = G^*)}{\Pr(\bR = R' \mid \bGs = G^*)}
    \le \frac{\binom{|N_{G^*}(W)| + t}{t}}{\binom{|N_{G^*}(W)|}{t}} \le \left(\frac{|N_{G^*}(W)|}{|N_{G^*}(W)| - t}\right)^t \le \left(\frac{1+\eps^2}{1-2\eps^2}\right)^t.
  \]
  In particular, for every $\cX \subseteq \binom{\br{n} \setminus W}{t}$,
  \begin{equation}
    \label{eq:hypergeom_close}
    \Pr(\bR \in \cX \mid \bGs) \le |\cX| \cdot \binom{n-k}{t}^{-1} \cdot e^{4\eps^2t}.
  \end{equation}
  
  Now, since $|\bG \triangle \bGs| \le |\bR||W| = kt = d$, \Cref{cor:core-d-nested} yields that $\bG \in \Core_{2d}^{r}(\alpha)$ implies that $\bGs \in \Core_{d}^{r}(\alpha)$ and that $\core_{d}^{r}(\bGs) \preceq \core_0^r(\bG)$.
  In particular, the event $\bG \in \Core_{2d}^r(\alpha) \wedge \cA_{W,t}$ implies the event
  \[
    |\bR \cap S^*| \le \frac{1-\eps/2}{r} \cdot t
    \quad
    \text{for some $S^* \in \core_d^r(\bGs)$},
  \]
  which we will denote by $\cA_{W,t}^*$.
  To summarise, we have
  \[
    \begin{split}
      \Pr(\bG \in \Core_{2d}^{r}(\alpha) \wedge \cA_{W,t}\mid \cY_W)
      & \le \Pr(\bGs \in \Core_d^r(\alpha) \wedge \cA_{W,t}^*) \\
      & \le \Pr(\cA_{W,t}^* \mid \bGs \in \Core_d^r(\alpha)).
    \end{split}
  \]
  Finally, we conclude from \eqref{eq:hypergeom_close} that
  \begin{multline*}
    \Pr(\cA^*_{W,t} \mid  \bGs \in \Core_{d}^r(\alpha)) \\
    \le e^{4\eps^2t} \max_{G^* \in \Core_d^r(\alpha)}
    \sum_{S^* \in \core_d^r(G^*)} \Pr\left(\text{Hyp}(n - k, |S^* \setminus W|, t) \le \frac{1-\eps/2}{r} \cdot t\right) \\
    \le e^{4\eps^2t} \cdot r \cdot \Pr\left(\text{Hyp}(n - k, (1/r-\alpha)n - k, t) \le \frac{1-\eps/2}{r} \cdot t\right) \le e^{-ct},
  \end{multline*}
  for some constant $c = c(r,\alpha,\eps) > 0$, provided that $\alpha \le \eps/(4r)$.
\end{proof}

\section{The threshold for Simonovits's theorem in $\Gnp$}
\label{sec:0-statement}
In this section, we prove \Cref{thm:0-statement-Simonovits}.
We start by setting some notational conventions.
As in \Cref{sec:rigidity}, given a collection $\Pi$ of pairwise-disjoint sets of vertices of a graph, we
denote by $\int(\Pi)$ the set of pairs of vertices contained in a single set of $\Pi$
and by $\ext(\Pi)$ the set of remaining pairs of vertices of $\bigcup \Pi$.
If $\Pi$ is a collection of sets of vertices of $K_n$, we also define
\[
  \ext^*(\Pi) \coloneqq K_n \setminus \int(\Pi)
\]
and note that $\ext^*(\Pi) \supseteq \ext(\Pi)$ and that $\ext^*(\Pi) = \ext(\Pi)$ whenever $\bigcup \Pi = K_n$.
Let $H$ be a nonbipartite, edge-critical, strictly $2$-balanced graph and denote $r \coloneqq \chi(H)-1$.
As in \cite{HosSam}, we denote by $\cH$ the collection of all copies of $H$ in
$K_n$, viewed as a hypergraph with vertex set $K_n$.
Further, for each $e \in K_n$, let
\[
  \partial_e \cH \coloneqq \{K \setminus e : e \in K \in \cH\}
\]
denote the copies of all proper, spanning subgraphs of $H$ that form a copy of $H$ together with the edge $e$.
Finally, we write $\partial_e\cH[G']$ to denote the subhypergraph of $\partial_e \cH$ that is induced by the subgraph $G' \subseteq K_n$.

Let $\bG \sim \Gnp$.
Suppose first that $1/n \ll p \ll n^{-1/m_2(H)}$.
It follows from the definition of $2$-density and Markov's inequality that whp $\bG$ contains only $o(n^2p)$ copies of $H$;
consequently, whp $\bG$ contains an $H$-free subgraph with $(1-o(1))\binom{n}{2}p$ edges.
On the other hand, \Cref{cor:equiv_classes_edges_count} implies that whp the largest size of an $r$-cut in $\bG$ is at most
\[
  \max_{\Pi : \text{$r$-cut}} |\ext(\Pi)| \cdot p + 3n\sqrt{np} = \left(1- \frac{1}{r} + o(1)\right) \binom{n}{2}p.
\]
Therefore, $\bG$ is whp not $H$-Simonovits.
We may thus assume from now on that
\[
  \eps n^{-1/m_2(H)} \le p \le (\theta_H-\eps) \cdot n^{-1/m_2(H)} \cdot (\log n)^{1/(e_H-1)}
\]
for some positive constant $\eps$.

It is clearly sufficient to prove that, with high probability, for some largest $r$-cut $\Pi$ of $\bG$ and an edge $e \in \int(\Pi) \cap \bG$, the graph $(\ext(\Pi) \cap \bG) \cup e$ is $H$-free.  We will prove a stronger statement.

Note that every largest $r$-cut $\Pi$ of $\bG$ satisfies
\[
  \ext(\Pi) = K_n \setminus \int(\Pi) \subseteq K_n \setminus \int(\core_0^r(\bG)) = \ext^*(\core_0^r(\bG)),
\]
provided that $\bG$ has a $(0,r)$-core.
The aforementioned stronger statement that implies \Cref{thm:0-statement-Simonovits} is that whp $\bG$ has a $(0,r)$-core with minimum part size $n/r-o(n)$ and there is an $e \in \int(\core_0^r(\bG)) \cap \bG$ such that $\partial_e\cH[\ext^*(\core_0^r(\bG)) \cap \bG]$ is empty;  note that this implies that $(\ext(\Pi) \cap \bG) \cup e$ is $H$-free for every largest $r$-cut $\Pi$ of $\bG$.
More precisely, set $\alpha \coloneqq (1/\log n)^2$ and define, for each $e \in K_n$, the event
\[
  \cY_e \coloneqq \big\{G \in \Core_0^r(\alpha) : e \in \int(\core_0^r(G)) \cap G \wedge \partial_e\cH[\ext^*(\core_0^r(G)) \cap G] = \emptyset\big\}.
\]
Our goal is to prove that whp $\bG \in \cY_e$ for some $e \in K_n$.
Denoting by $Z$ the number of $e \in K_n$ satisfying $\bG \in \cY_e$, it will be enough to show that
\begin{equation}
  \label{eq:0-statement-goal}
  \Ex[Z^2] \le (1+o(1)) \cdot \Ex[Z]^2.
\end{equation}
Indeed, if \eqref{eq:0-statement-goal} holds, then
\[
  \Pr(\text{$\bG \in \cY_e$ for some $e \in K_n$}) = \Pr(Z \neq 0) \ge \frac{\Ex[Z]^2}{\Ex[Z^2]} = 1-o(1)
\]
by the Paley--Zygmund inequality.

\subsection*{Proof outline}

In order to establish \eqref{eq:0-statement-goal}, we separately prove a lower bound on $\Ex[Z]$ and an upper bound on $\Ex[Z^2]$.  We obtain a lower bound on $\Ex[Z]$ using a delicate switching argument that (roughly speaking) goes as follows.  We first choose $d \gg \log n$ so that whp $\bG \in \Core_{2d}^r(\alpha)$;
this is possible thanks to \Cref{cor:core}.
Fix some $e \in K_n$ and assume that $e \in \int(\core_{2d}^r(\bG)) \cap \bG$.  Our upper-bound assumption on $p$ and the fact that $H$ is strictly $2$-balanced imply that whp $\partial_e\cH[\bG]$ is a matching of size $O(\log n)$, which in turn allows us to analyse the following `resampling' process:  Remove from $\bG$ all the edges of $\bigcup \partial_e \cH[\bG]$, denote the resulting graph by $\bGs$, and consider the conditional distribution of $\bG$ given $\bGs$.  The fact that $\partial_e\cH[\bG]$ is a matching with $o(d)$ edges allows us to infer that $\bGs \in \Core_{d}^r(\alpha)$ and to essentially couple the conditional distribution of $\partial_e\cH[\bG]$ given $\bGs$ with a $p^{e_H-1}$-random subset of $\partial_e \cH$, giving
\[
  \Pr\big(\partial_e\cH[\ext^*(\core_d^r(\bGs)) \cap \bG] = \emptyset \mid \bGs\big) \ge (1-o(1)) \cdot \left(1-p^{e_H-1}\right)^{|\partial_e\cH[\ext^*(\core_d^r(\bGs))]|}.
\]
Since $\core_d^r(\bGs)$ is a collection of $r$ pairwise-disjoint sets of size at least $n/r - \alpha n$ each, we have
\[
  |\partial_e \cH[\ext^*(\core_d^r(\bGs))]| \le (1+O(\alpha)) \cdot \Cop\big(H, K_r(n/r)^+\big) = (1+O(\alpha)) \cdot \pi_H \cdot (n/r)^{v_H-2}.
\]
Further, since $\bG$ and $\bGs$ differ in $o(d)$ edges, we have $\core_0^r(\bG) \succeq \core_d^r(\bGs)$, by \Cref{cor:core-d-nested}, and thus $\ext^*(\core_0^r(\bG)) \subseteq \ext^*(\core_d^r(\bGs))$; consequently, since $n^{v_H-2} p^{e_H-1} = O(\log n)$ and $\alpha \ll 1/\log n$,
\begin{multline*}
  \Pr\big(\partial_e\cH[\ext^*(\core_0^r(\bG)) \cap \bG] = \emptyset \mid \bG \in \Core_{2d}^r(\alpha) \wedge e \in \int(\core_{2d}^r(\bG)) \cap \bG\big) \\
  \ge \exp\left(-\pi_H \cdot (n/r)^{v_H-2} \cdot p^{e_H-1}-o(1)\right).
\end{multline*}
A lower bound on $\Ex[Z]$ now follows by multiplying the above inequality by the probability of the event in the conditioning and summing the result over all $e \in K_n$.

In order to prove an upper bound on $\Ex[Z^2]$, we adapt an elegant argument of DeMarco--Kahn~\cite{DeMKah15Tur}, which allows us to bound, for every pair $e, f$ of edges of $K_n$, the conditional probability
\[
  \Pr\big((\partial_e\cH \cup \partial_f\cH)[\ext(\core_0^r(\bG)) \cap \bG] = \emptyset \mid \bG \in \Core_0^r(\alpha) \wedge e,f \in \int(\core_0^r(\bG)) \cap \bG\big)
\]
from above by the (unconditional) probability of the same event with $\core_0^r(\bG)$ replaced by a fixed collection of $r$ pairwise-disjoint sets of at least $n/r - \alpha n$ vertices each.  The latter probability can be easily shown, using Janson's inequality, to be at most $\exp\big(-2\pi_H \cdot (n/r)^{v_H-2} \cdot p^{e_H-1} + o(1)\big)$.  An upper bound on $\Ex[Z^2]$ is then deduced in a straightforward manner by summing the above estimate over all pairs $e, f$.

\subsection*{Organisation}

The remainder of this section is organised as follows.  In \cref{sec:preliminaries}, we recall the statement of Janson's inequality, prove several useful estimates concerning the hypergraph $\partial_e\cH$, and derive estimates on the moments of $|\int(\core_d^r(\bG)) \cap \bG| \cdot \1_{\bG \in \Core_d^r(\alpha)}$ from \cref{cor:core}.  In the remaining two sections, we prove the lower bound on $\Ex[Z]$ and the upper bound on $\Ex[Z^2]$.

\subsection{Preliminaries}
\label{sec:preliminaries}

Given a set $V$ and a real $p \in [0,1]$, we denote by $V_p$ the random subset of $V$ obtained by independently retaining each element of $V$ with probability $p$.  Further, given a hypergraph $\cG$ with vertex set $V$, we define the following two quantities:
\[
  \mu_p(\cG) \coloneqq \sum_{A \in \cG} p^{|A|}
  \qquad
  \text{and}
  \qquad
  \Delta_p(\cG) \coloneqq \sum_{\substack{A, B \in \cG \\ A \neq B, A \cap B \neq \emptyset}} p^{|A \cup B|},
\]
where the second sum is over unordered pairs of edges;  in other words, $\mu_p(\cG)$ is just the expected number of edges of $\cG[V_p]$ and $\Delta_p(\cG)$ is the expected number of pairs of distinct edges of $\cG[V_p]$ that intersect.  The following well-known inequality of Janson plays a key role in our proof of the upper bound on $\Ex[Z^2]$.

\begin{thm}[{Janson's inequality~\cite{Jan90}}]
  \label{thm:Janson-inequality}
  Let $\cG$ be a hypergraph on a finite vertex set~$V$.  For all $p \in [0,1]$,
  \[
    \Pr(\cG[V_p] = \emptyset) \le \exp\big(-\mu_p(\cG)+\Delta_p(\cG)\big).
  \]
\end{thm}

We will now prove lower and upper bounds on the values of $\mu_p$ and $\Delta_p$ on certain induced subhypergraphs of $\partial_e \cH$ and  $\partial_e \cH \cup \partial_f \cH$.  We start with an estimate on the sizes of subhypergraphs of $\partial_e\cH$ induced by graphs that are close to a complete, balanced, $r$-partite graph.

\begin{lemma}
  \label{lemma:copies-of-H-ext}
  Let $\alpha$ be a nonnegative real and suppose that $\cC$ is a family of $r$ pairwise-disjoint subsets of $\br{n}$ such that $|X| \ge n/r - \alpha n$ for each $X \in \cC$.
  There is a constant $C_H$ that depends only on $H$ such that, for every $e \in \int(\cC)$,
  \begin{align*}
    |\partial_e \cH [\ext(\cC)]| & \ge \big(\pi_H - C_H\alpha \big) \cdot (n/r)^{v_H-2} - C_Hn^{v_H-3}, \\
    |\partial_e \cH [\ext^*(\cC)]| & \le \big(\pi_H + C_H\alpha\big) \cdot (n/r)^{v_H-2}.
  \end{align*}
\end{lemma}
\begin{proof}
  Since $\ext(\cC) \supseteq K_r(n/r-\alpha n)$ by our assumption on $\cC$, when $e \in \int(\cC)$, we have
  \[
    |\partial_e \cH [\ext(\cC)]| \ge \Cop\big(H, K_r(n/r-\alpha n)^+\big) \ge \pi_H \cdot (n/r-\alpha n)^{v_H-2} - O(n^{v_H-3}),
  \]
  which implies the first inequality.  (The reason why we may write such explicit error term is that $\Cop(H, K_r(m)^+)$ is a polynomial of degree $v_H-2$ in $m$.)
  Further, since\footnote{We write $G_1 \vee G_2$ for the graph obtained from the disjoint union of $G_1$ and $G_2$ by adding all edges joining $V(G_1)$ and $V(G_2)$.}
  \[
    \ext^*(\cC) \subseteq K_r(n/r-\alpha n) \vee K_{r\alpha n},
  \]
  every copy of $H$ minus an edge in $\ext^*(\cC)$ that is not fully contained in $K_r(n/r-\alpha n)$ must have at least one vertex in $K_{r\alpha n}$.  Consequently, there is a constant $C_H'$ that depends only on $H$ such that, for each $e \in \int(\cC)$, 
  \[
    |\partial_e\cH[\ext^*(\cC)]| \le \Cop\big(H, K_r(n/r)^+\big) + C_H' r\alpha n \cdot n^{v_H-3} \le \pi_H \cdot (n/r)^{v_H-2} + C_H'r\alpha n^{v_H-2},
  \]
  which implies the second inequality.
\end{proof}

Our second lemma supplies an upper bound on $\Delta_p(\partial_e \cH \cup \partial_f \cH)$, and thus also on $\Delta_p(\partial_e \cH)$.  Even though this upper bound is implicit in \cite[Lemma~5.2]{HosSam}, we include a (self-contained) proof here for completeness.

\begin{lemma}
  \label{lemma:Delta-p-partial-H}
  For every pair of distinct edges $e, f \in K_n$ and all $p \ge \eps n^{-1/m_2(H)}$,
  \[
    \Delta_p(\partial_e\cH \cup \partial_f \cH) \le Cn^{-\lambda} \cdot \left(n^{v_H-2} p^{e_H-1}\right)^2
  \]
  for some positive $\lambda = \lambda(H)$ and $C = C(H,\eps)$.
\end{lemma}

The following non-probabilistic inequality, which is \cite[Lemma~4.1]{HosSam}, encapsulates the key inequality in the proof of \cref{lemma:Delta-p-partial-H}.

\begin{lemma}[\cite{HosSam}]
  \label{lemma:2-balanced-condition}
  Let $H$ be a nonempty graph and suppose that $p \ge \eps n^{-1/m_2(H)}$ for some $\eps > 0$.
  Then, $n^{v_{H'}-2} p^{e_{H'}-1} \ge \eps^{e_{H'} - 1}$ for every nonempty subgraph $H' \subseteq H$.  Moreover, if $H$ is strictly $2$-balanced, then there exists a $\lambda > 0$ that depends only on $H$ such that $n^{v_{H'}-2} p^{e_{H'}-1} \ge \eps^{e_{H'} - 1} n^{\lambda}$ for every $H' \subseteq H$ with $1 < e_{H'} < e_H$.
\end{lemma}
\begin{proof}
  Let $H'$ be a nonempty subgraph of $H$.  Since the assertion of the lemma holds vacuously if $e_{H'} = 1$, we may assume that $e_{H'}>1$.  By our assumption on $p$,
  \[
    n^{v_{H'}-2} p^{e_{H'}-1} \ge n^{v_{H'}-2} \left(\eps n^{-\frac{1}{m_2(H)}}\right)^{e_{H'}-1} = \eps^{e_{H'} - 1} n^{v_{H'}-2 - \frac{e_{H'}-1}{m_2(H)}}.
  \]
  Now, recall that the definition of $m_2(H)$ implies that $v_{H'} - 2 \ge \frac{e_{H'}-1}{m_2(H)}$ and that, when $H$ is strictly $2$-balanced, this inequality is strict unless $H' = H$. Thus, the exponent of $n$ is nonnegative and it is positive if $H$ is strictly $2$-balanced and $1 < e_{H'} < e_H$.
\end{proof}

\begin{proof}[{Proof of~\cref{lemma:Delta-p-partial-H}}]
  Observe first that
  \[
    \Delta_p(\partial_e \cH \cup \partial_f \cH) \le \Delta_p(\partial_e \cH) + \Delta_p(\partial_f \cH) + \Delta_p(\partial_e \cH, \partial_f \cH),
  \]
  where
  \[
    \Delta_p(\partial_e \cH, \partial_f \cH) \coloneqq \sum_{K \in \partial_e \cH} \sum_{\substack{K' \in \partial_f \cH \\ K \cap K' \neq \emptyset}} p^{e(K \cup K')}.
  \]
  Now, for some positive $\lambda = \lambda(H)$, $C' = C'(H)$, and $C = C(H, \eps)$
  \[
    \begin{split}
      2\Delta_p(\partial_e \cH) & = \sum_{K \in \partial_e \cH} \sum_{\substack{K' \in \partial_e \cH \setminus \{K\} \\ K \cap K' \neq \emptyset}} p^{e(K \cup K')} = \sum_{\substack{J \subseteq H \\ 2 \le e_J < e_H}} \sum_{K \in \partial_e \cH} \sum_{\substack{K' \in \partial_e \cH \setminus \{K\} \\ (K \cap K') \cup e \cong J}} p^{2e_H - e_J - 1} \\
      & \le C' \sum_{\substack{J \subseteq H \\ 2 \le e_J < e_H}} n^{2v_H-v_J-2} p^{2e_H-e_J-1} = C' \sum_{\substack{J \subseteq H \\ 2 \le e_J < e_H}} \frac{n^2p}{n^{v_J}p^{e_J}} \cdot \left(n^{v_H-2} p^{e_H-1}\right)^2 \\
      & \le Cn^{-\lambda} \cdot \left(n^{v_H-2} p^{e_H-1}\right)^2,
    \end{split}
  \]
  where the last inequality follows from~\cref{lemma:2-balanced-condition};  this estimate clearly remains true when we replace $e$ with $f$.  Similarly, for some $C' = C'(H)$ and $C = C(H, \eps)$,
  \[
    \begin{split}
      \Delta_p(\partial_e \cH, \partial_f \cH) & = \sum_{\substack{J \subseteq H \\ 1 \le e_J < e_H-1}} \sum_{K \in \partial_e \cH} \sum_{\substack{K' \in \partial_f \cH \setminus \{K\} \\ K \cap K' \cong J}} p^{2e_H - e_J - 2} \\
      & \le C' \sum_{\substack{J \subseteq H \\ 1 \le e_J < e_H}} \frac{n^{|e \cap f|}}{n^{v_J}p^{e_J}} \cdot \left(n^{v_H-2} p^{e_H-1}\right)^2
      \le \frac{C}{np} \cdot \left(n^{v_H-2} p^{e_H-1}\right)^2,
    \end{split}
  \]
  where the last inequality again follows from~\cref{lemma:2-balanced-condition}.  Since $np \ge \eps n^{1-1/m_2(H)}$ and $m_2(H) > 1$, the claimed inequality follows.
\end{proof}

\begin{lemma}
  \label{lemma:edges-int-core}
  Suppose that $\alpha, p \in (0,1/2)$ and a nonnegative integer $d$ satisfy
  \[
    \alpha \ll 1
    \qquad
    \text{and}
    \qquad
    pn \gg \max\left\{\alpha^{-4}, (d/\alpha)^2\right\}.
  \]
  Then, for all fixed $k \ge 0$ and $r \ge 2$, the random graph $\bG \sim \Gnp$ satisfies
  \[
    \Ex\left[\left|\int(\core_d^r(\bG)) \cap \bG\right|^k \cdot \1_{\bG \in \Core_d^r(\alpha)}\right] = (1+o(1)) \cdot \left(\frac{n^2p}{2r}\right)^k.
  \]
\end{lemma}
\begin{proof}
  Set $m \coloneqq \binom{n}{2}p$ and note that our assumption that $np \gg 1$ guarantees that whp $|\bG| \in [m/2, 3m/2]$.
  In particular, the assumed asymptotic relations between $\alpha$, $p$, and $d$ allow us to conclude from \cref{cor:core} that whp $\bG \in \Core_d^r(\alpha)$.  Since for every $G \in \Core_d^r(\alpha)$, the graph $\core_d^r(G)$ is a disjoint union of $r$ complete graphs of order at least $n/r - \alpha n$ each, we have
  \[
    \left|\int(\core_{d}^r(G))\right| = \frac{n^2}{2r} \pm O(\alpha n^2)
  \]
  and further, by the Chernoff bound and the union bound over the at most $(r+1)^n$ possible graphs $\core_{d}^r(G)$,
  \[
    \Pr\left(\forall G \in \Core_d^r(\alpha) \;\; \left||\int(\core_{d}^r(G)) \cap \bG| - \frac{n^2p}{2r}\right| = O\left(\alpha + (np)^{-1/2}\right) \cdot n^2p \right) \ge 1 - e^{-n}.
  \]
  The assertion of the lemma follows, as $\alpha + (np)^{-1/2} \ll 1$ by our assumptions.
\end{proof}

\subsection{Proof of the lower bound on $\Ex[Z]$}
\label{sec:proof-lower-bound-Z}

The following lemma, which is the main technical step in the proof of the lower bound on $\Ex[Z]$, abstracts the essence of the `resampling' procedure that we described in the proof outline presented above.  Given a hypergraph $\cG$, we denote by $\cI(\cG)$ the family of its independent sets.

\begin{lemma}
  \label{lemma:matching-resample}
  Suppose that $\cG$ is a $k$-uniform hypergraph on $V$, let $t$ be a nonnegative integer, and define
  \[
    \cM \coloneqq \{R \subseteq V : \text{$\cG[R]$ is a matching with $\le t$ edges}\}.
  \]
  Suppose further that $p \in (0,1)$, let $\bR \sim V_p$, and let $\bRs \coloneqq \bR \setminus \bigcup \cG[\bR]$.  Then, for all $\cA \subseteq \cI(\cG)$ and $A \colon \cA \to \cP(V)$, letting
  \[
    \frac{q}{1-q} \coloneqq \left(\frac{p}{1-p}\right)^k
    \qquad
    \text{and}
    \qquad
    a \coloneqq \max\big\{|\cG[A(R^*)]| : R^* \in \cA\big\},
  \]
  we have
  \[
    \Pr\big(\bR \in \cM \wedge \bRs \in \cA \wedge \cG[A(\bRs) \cap \bR] = \emptyset\big) \ge
    (1-q)^{a} \cdot \Pr(\bR \in \cM \wedge \bRs \in \cA).
  \]
\end{lemma}
\begin{proof}
  Define the map $\cS \colon \cA \to \cP(\cP(V))$ by
  \[
    \cS(R^*) \coloneqq \big\{S \subseteq V \setminus R^* : \cG[S] \text{ is a matching spanning $S$ and } |\cG[R^* \cup S]| = |\cG[S]| \le t\big\}.
  \]
  Since, for every $R^* \in \cA$, the family $\cS(R^*)$ comprises precisely those sets $S \subseteq V \setminus R^*$ that are unions of edges of $\cG$ and satisfy $R^* \cup S \in \cM$, we have
  \[
    P_{R^*}(S) \coloneqq \frac{\Pr(\bR = R^* \cup S \wedge \bRs = R^*)}{\Pr(\bR \in \cM \wedge \bRs = R^*)} = \frac{1}{Z_{R^*}} \cdot \left(\frac{p}{1-p}\right)^{|S|} = \frac{1}{Z_{R^*}} \cdot \left(\frac{p}{1-p}\right)^{k \cdot |\cG[S]|}
  \]
  for all $S \in \cS(R^*)$, where
  \[
    Z_{R^*} \coloneqq \sum_{S \in \cS(R^*)} \left(\frac{p}{1-p}\right)^{|S|}
    = \sum_{S \in \cS(R^*)} \left(\frac{q}{1-q}\right)^{|\cG[S]|}
    = \sum_{\substack{M \subseteq \cG \\ \bigcup M \in \cS(R^*)}} \left(\frac{q}{1-q}\right)^{|M|}
    .
  \]
  In other words, letting $\bM$ denote the $q$-random subset of $\cG$, we have
  \[
    P_{R^*}(S) = \Pr\left(\bM = \cG[S] \mid \bigcup \bM \in \cS(R^*)\right).
  \]
  Since both $\big\{M \subseteq \cG : \bigcup M \in \cS(R^*) \big\}$ and $\big\{M \subseteq \cG : M[A(R^*)] = \emptyset\big\}$ are decreasing families of subsets of $\cG$, Harris's inequality~\cite{Har60} gives
  \begin{multline*}
    \frac{\Pr(\bR \in \cM \wedge \bRs = R^* \wedge \cG[A(R^*) \cap \bR] = \emptyset)}{\Pr(\bR \in \cM \wedge \bRs = R^*)}
    = \sum_{\substack{S \in \cS(R^*) \\ \cG[A(R^*) \cap S] = \emptyset}} P_{R^*}(S) \\
    = \Pr\big(\bM[A(R^*)] = \emptyset \mid \bigcup \bM \in \cS(R^*)\big) \ge \Pr\big(\bM[A(R^*)] = \emptyset\big) \\
    = (1-q)^{|\cG[A(R^*)]|} \ge (1-q)^a.
  \end{multline*}
  Multiplying the above inequality through by $\Pr(\bR \in \cM \wedge \bRs = R^*)$ and summing the result over all $R^* \in \cA$ gives the desired inequality.
\end{proof}

\begin{cor}
  \label{cor:Pr-Ye-lower}
  Suppose that $p \le C n^{-1/m_2(H)} (\log n)^{1/(e_H-1)}$ for some constant $C$ and let $\bG \sim \Gnp$.
  For every $e \in K_n$ and all $\alpha > 0$, letting $d \coloneqq \left\lceil \log n \right\rceil^2$, we have
  \begin{multline*}
    \Pr(\cY_e) \ge \exp\left(- \big(\pi_H + O(\alpha+p)\big) \cdot (n/r)^{v_H-2}p^{e_H-1}\right) \\
    \cdot \left(\Pr\big(\bG \in \Core_{2d}^r(\alpha) \wedge e \in \int(\core_{2d}^r(\bG)) \cap \bG\big) - o(p)\right).
  \end{multline*}
\end{cor}
\begin{proof}
  Let $\omega \coloneqq \NN \to \RR$ be an arbitrary function satisfying $1 \ll \omega(n) \ll \log n$ and let
  \[
    t \coloneqq \omega(n) \cdot \Ex|\partial_e\cH[\bG]|.
  \]
  It follows from our upper-bound assumption on $p$ that, for some constants $C_1 = C_1(H)$ and $C_2 = C_2(H,C)$, we have
  \begin{equation}
    \label{eq:t-upper}
    t \le \omega(n) \cdot C_1n^{v_H-2}p^{e_H-1} \le C_2 \omega(n) \log n \le d/e_H,
  \end{equation}
  provided that $n$ is sufficiently large.
  As in \cref{lemma:matching-resample}, let
  \[
    \cM \coloneqq \big\{G \subseteq K_n : \partial_e\cH[G] \text{ is a matching with $\le t$ edges}\big\}.
  \]
  Further, let $\bGs \coloneqq \bG \setminus \bigcup \partial_e \cH[\bG]$ and define
  \[
    \cA \coloneqq \big\{G^* \subseteq K_n : G^* \in \Core_{d}^r(\alpha) \wedge e \in \int(\core_d^r(G^*)) \cap G^* \wedge \partial_e\cH[G^*] = \emptyset\big\}
  \]
  and the function $A \colon \cA \to \cP(K_n)$ by $A(G^*) \coloneqq \ext^*(\core_d^r(G^*))$ for every $G^* \in \cA$.  By \cref{lemma:copies-of-H-ext}, for some constant $C_H$ that depends only on $H$,
  \[
    \max_{G^* \in \cA}|\partial_e\cH[A(G^*)]| \le \big(\pi_H + C_H\alpha\big) \cdot (n/r)^{v_H-2} \eqqcolon a.
  \]
  Further, since $1-x \ge \exp(-x/(1-x))$ for all $x \in (0,1)$, \cref{lemma:matching-resample} applied to the $(e_H-1)$-uniform hypergraph $\partial_e \cH$ with vertex set $K_n$ yields
  \begin{multline}
    \label{eq:no-H-on-e-lower-bound}
    \Pr\big(\bG \in \cM \wedge \bGs \in \cA \wedge \partial_e\cH[A(\bGs) \cap \bG] = \emptyset\big) \\
    \ge \exp\left(-\left(\frac{p}{1-p}\right)^{e_H-1} \cdot a\right) \cdot \Pr\big(\bG \in \cM \wedge \bGs \in \cA\big).
  \end{multline}
  We now show that~\eqref{eq:no-H-on-e-lower-bound} implies the assertion of the corollary.
  
  First, since $(1-x)^{1-e_H} = 1 + (e_H-1)x + O(x^2)$ as $x \to 0$, we have
  \[
    \exp\left(-\left(\frac{p}{1-p}\right)^{e_H-1} \cdot a\right) \ge \exp\left(- \big(\pi_H + C_H\alpha\big) \cdot (1+e_Hp) \cdot (n/r)^{v_H-2} p^{e_H-1}\right).
  \]
  Second, observe that $\bG \in \cM$ implies that
  \begin{equation}
    \label{eq:G-Gs-on-cM}
    |\bGs| = |\bG| - (e_H-1) \cdot |\partial_e\cH[\bG]| \ge |\bG| - e_Ht \ge |\bG| - d,
  \end{equation}
  see~\eqref{eq:t-upper}.
  Consequently, \cref{cor:core-d-nested} implies that the event $\bG \in \Core_{2d}^r(\alpha) \cap \cM$ is contained in the event $\bGs \in \Core_{d}^r(\alpha)$ and $\int(\core_{2d}^r(\bG)) \subseteq \int(\core_d^r(\bGs))$.  We thus have
  \begin{multline*}
    \Pr\big(\bG \in \cM \wedge \bGs \in \cA\big) \ge \Pr\big(\bG \in \Core_{2d}^r(\alpha) \cap \cM \wedge e \in \int(\core_{2d}^r(\bG)) \cap \bG\big) \\
    \ge \Pr\big(\bG \in \Core_{2d}^r(\alpha) \wedge e \in \int(\core_{2d}^r(\bG)) \cap \bG\big) - \Pr(\bG \notin \cM \wedge e \in \bG).
  \end{multline*}
  Since the events $e \in \bG$ and $\bG \notin \cM$ are independent, we further have
  \[
    \Pr(\bG \notin \cM \wedge e \in \bG) = p \cdot \Pr(\bG \notin \cM) \le p \cdot \left( \Pr(|\partial_e \cH[\bG]| > t) + \Pr\big(\Delta(\partial_e\cH[\bG]) \ge 2\big)\right).
  \]
  The first probability in the right-hand side is at most $1/\omega(n)$, by Markov's inequality and the definition of $t$, whereas the second probability can be bounded using \cref{lemma:Delta-p-partial-H} as follows:
  \[
    \begin{split}
      \Pr\big(\Delta(\partial_e\cH[\bG]) \ge 2\big) & \le \Ex|\{(K,K') \in (\partial_e\cH[\bG])^2 : K \neq K', K \cap K' \neq \emptyset\}| \\
      & = 2\Delta_p(\partial_e \cH) \le Cn^{-\lambda} \cdot \left(n^{v_H-2}p^{e_H-1}\right)^2
    \end{split}
  \]
  for some positive $C=C(H,\eps)$ and $\lambda = \lambda(H)$;  since $n^{v_H-2}p^{e_H-1} = O(\log n)$ under our upper-bound assumption on $p$, we may conclude that $\Pr\big(\Delta(\partial_e\cH[\bG]) \ge 2\big) = o(1)$.
  
  Finally, since $\bG \in \cM$ implies that $|\bG| \le |\bGs| + d$, see~\eqref{eq:G-Gs-on-cM}, \cref{cor:core-d-nested} implies that the event  $\bG \in \cM \wedge \bGs \in \Core_{d}^r(\alpha)$ is contained in the event that $\bG \in \Core_{0}^r(\alpha)$ and $\int(\core_d^r(\bGs)) \subseteq \int(\core_0^r(\bG))$ (equivalently, that $\ext^*(\core_0^r(\bG)) \subseteq \ext^*(\core_d^r(\bGs)) = A(\bGs)$).  Further, since $\bGs \in \cA$ implies that $e \in \int(\core_d^r(\bGs)) \cap \bGs$, we conclude that
  \[
    \Pr(\cY_e) \ge \Pr\big(\bG \in \cM \wedge \bGs \in \cA \wedge \partial_e\cH[A(\bGs) \cap \bG] = \emptyset\big).
  \]
  The assertion of the lemma follows by combining the above inequality with~\eqref{eq:no-H-on-e-lower-bound} and the lower bounds on the two terms in the right-hand side of~\eqref{eq:no-H-on-e-lower-bound}.
\end{proof}

We are finally ready to complete the derivation of the lower bound on $\Ex[Z]$.
Since $n^{v_H-2}p^{e_H-1} = O(\log n)$, we have
\[
  \big(\pi_H + O(\alpha+p)\big) \cdot (n/r)^{v_H-2}p^{e_H-1} = \pi_H \cdot (n/r)^{v_H-2} p^{e_H-1} + o(1).
\]
Consequently, we may deduce from~\Cref{cor:Pr-Ye-lower} that
\begin{multline*}
  \Ex[Z] = \sum_{e \in K_n} \Pr(\cY_e) \ge \exp\left(-\pi_H \cdot (n/r)^{v_H-2} p^{e_H-1} - o(1)\right) \\
  \cdot \left(\Ex\left[\left|\int(\core_{2d}^r(\bG)) \cap \bG \right| \cdot \1_{\bG \in \Core_{2d}^r(\alpha)}\right]-o(n^2p)\right),
\end{multline*}
where $d \coloneqq \left\lceil \log n\right\rceil^2$.
Finally, \Cref{lemma:edges-int-core} allows us to conclude that
\begin{equation}
  \label{eq:Ex-Z-lower}
  \Ex[Z] \ge (1+o(1)) \cdot \exp\left(-\pi_H \cdot (n/r)^{v_H-2} p^{e_H-1}\right) \cdot \frac{n^2p}{2r}.
\end{equation}

\subsection{Proof of the upper bound on $\Ex[Z^2]$}
\label{sec:proof-upper-bound-Z}

Given distinct edges $e, f \in K_n$ and a family $\cC$ of $r$ pairwise-disjoint subsets of $\br{n}$, define
\begin{align*}
  I_{e,f}(\cC) & \coloneqq \big\{G \subseteq K_n : e,f \in \int(\cC) \cap G\big\}, \\
  E_{e,f}(\cC) & \coloneqq \big\{G \subseteq K_n : (\partial_e\cH \cup \partial_f\cH)[\ext(\cC) \cap G] = \emptyset\big\}.
\end{align*}
and note that, for every graph $G \subseteq K_n$,
\[
  G \in \cY_e \cap \cY_f \quad \Longrightarrow \quad G \in \Core_0^r(\alpha) \cap I_{e,f}(\core_0^r(G)) \cap E_{e,f}(\core_0^r(G)),
\]
as $\ext(\cC) \subseteq \ext^*(\cC)$ for every family $\cC$.  We may thus conclude that
\begin{equation}
  \label{eq:Ex-Z-squared}
  \Ex[Z^2] \le \Ex[Z] + \sum_{\substack{e, f \in K_n \\ e \neq f}} \Pr\big(\bG \in \Core_0^r(\alpha) \cap I_{e,f}(\core_0^r(\bG)) \cap E_{e,f}(\core_0^r(\bG))\big).
\end{equation}
Our next lemma, which is a variant of \cite[Lemma~10.2]{DeMKah15Tur}, will allow us to bound from above the probabilities in the right-hand side of~\eqref{eq:Ex-Z-squared}.

\begin{lemma}
  \label{lemma:correlation-argument}
  Let $\alpha$ be a nonnegative real and let $\fC$ be the collection of all $r$-element families $\cC$ of pairwise-disjoint subsets of $\br{n}$ satisfying $|X| \ge n/r-\alpha n$ for all $X \in \cC$.
  Suppose that, for each $\cC \in \fC$, we have an event $I(\cC)$ that is determined by $\int(\cC)$ and an event $E(\cC)$ that is determined by $\ext(\cC)$ and decreasing, and satisfies $\Pr(\bG \in E(\cC)) \le \xi$.  Then,
  \[
    \Pr\big(\bG \in \Core_{0}^r(\alpha) \cap I(\core_0^r(\bG)) \cap E(\core_0^r(\bG))\big) \le \xi \cdot \Pr\big(\bG \in \Core_{0}^r(\alpha) \cap I(\core_0^r(\bG))\big).
  \]
\end{lemma}
\begin{proof}
  For any $\cC \in \fC$, denote by $\Core(\cC)$ the family of all $G \subseteq K_n$ such that $\core_0^r(G) = \cC$.
  We first observe that $\Core(\cC)$ is increasing in $\ext(\cC)$. Indeed, since adding to $G$ an edge of $\ext(\core_0^r(G))$ does not change the set of largest $r$-cuts, the resulting graph has the same core as $G$.
  Let $\cF_\cC$ be the $\sigma$-algebra generated by $(K_n \setminus \ext(\cC)) \cap \bG$.
  Since $I(\cC)$ is determined by $\int(\cC) \subseteq K_n \setminus \ext(\cC)$, we have
  \begin{equation}
    \label{eq:correlation-conditioning}
    \Pr\big(\bG \in \Core(\cC) \cap I(\cC) \cap E(\cC)\mid \cF_\cC\big) = \Pr\big(\bG \in \Core(\cC) \cap E(\cC) \mid \cF_\cC\big) \cdot \1_{\bG \in I(\cC)}.
  \end{equation}
  Further, since $\Core(\cC)$ is increasing in $\ext(\cC)$ and $E(\cC)$ is decreasing in $\ext(\cC)$, Harris's inequality~\cite{Har60} implies that
  \begin{equation}
    \label{eq:correlation-Harris}
    \begin{split}
      \Pr\big(\bG \in \Core(\cC) \cap E(\cC) \mid \cF_\cC\big)
      & \le \Pr\big(\bG \in \Core(\cC) \mid \cF_\cC\big) \cdot \Pr\big(\bG \in E(\cC) \mid \cF_\cC\big) \\
      & = \Pr\big(\bG \in \Core(\cC) \mid \cF_\cC\big) \cdot \Pr\big(\bG \in E(\cC)\big),
    \end{split}
  \end{equation}
  where the equality holds as $E(\cC)$ is determined by $\ext(\cC)$ and thus the event $\bG \in E(\cC)$ is independent of $\cF_\cC$.  Substituting~\eqref{eq:correlation-Harris} into~\eqref{eq:correlation-conditioning}, and using our assumption, yields
  \[
    \begin{split}
      \Pr\big(\bG \in \Core(\cC) \cap I(\cC) \cap E(\cC)\mid \cF_\cC\big)
      & \le \xi \cdot \Pr\big(\bG \in \Core(\cC) \mid \cF_\cC\big) \cdot \1_{\bG \in I(\cC)}\\
      & = \xi \cdot \Pr\big(\bG \in \Core(\cC) \cap I(\cC) \mid \cF_\cC\big).
  \end{split}
  \]
  Consequently, since $\{\Core(\cC) : \cC \in \fC\}$ is a partition of $\Core_{0}^r(\alpha)$,
  \begin{multline*}
    \Pr\big(\bG \in \Core_{0}^r(\alpha) \cap I(\core_0^r(\bG)) \cap E(\core_0^r(\bG))\big) \\
    = \sum_{\cC \in \fC} \Pr\big(\bG \in \Core(\cC) \cap I(\cC) \cap E(\cC)\big) 
    \le \xi \cdot \sum_{\cC \in \fC} \Pr\big(\bG \in \Core(\cC) \cap I(\cC)\big) \\
    = \xi \cdot \Pr\big(\bG \in \Core_{0}^r(\alpha) \cap I(\core_0^r(\bG))\big),
  \end{multline*}
  as desired.
\end{proof}

Returning to~\eqref{eq:Ex-Z-squared}, since clearly $I_{e,f}(\cC)$ is determined by $\int(\cC)$ whereas $E_{e,f}(\cC)$ is determined by $\ext(\cC)$ and decreasing, \Cref{lemma:correlation-argument} implies that
\begin{equation}
  \label{eq:Ex-Z-2-upper}
  \begin{split}
    \Ex[Z^2] & \le \Ex[Z] + \xi \cdot \sum_{\substack{e,f \in K_n \\ e \neq f}} \Pr\big(\bG \in \Core_{0}^r(\alpha) \cap I_{e,f}(\core_0^r(\bG))\big) \\
    & \le \Ex[Z] + \xi \cdot \Ex\left[\left|\int(\core_0^r(\bG)) \cap \bG\right|^2 \cdot \1_{\bG \in \Core_{0}^r(\alpha)}\right],
  \end{split}
\end{equation}
where (writing $\fC$ for the collection of all $r$-element families $\cC$ of pairwise-disjoint subsets of $\br{n}$ satisfying $|X| \ge n/r - \alpha n$ for all $X \in \cC$)
\[
  \xi \coloneqq 
  \max_{\cC \in \fC} \Pr\big(\bG \in E_{e,f}(\cC)\big)
  = \max_{\cC \in \fC} \Pr\big((\partial_e\cH \cup \partial_f\cH)[\ext(\cC) \cap \bG] = \emptyset\big).
\]
It follows from \Cref{lemma:copies-of-H-ext} that, for every $\cC \in \fC$,
\[
  \begin{split}
    \mu_p\big((\partial_e \cH \cup \partial_f \cH)[\ext(\cC)]\big)
    & = \mu_p(\partial_e \cH[\ext(\cC)]) + \mu_p(\partial_f \cH[\ext(\cC)]) \\
    & \ge 2 \big(\pi_H - O(\alpha + n^{-1})\big) \cdot (n/r)^{v_H-2} p^{e_H-1} \\
    & \ge 2\pi_H \cdot (n/r)^{v_H-2} p^{e_H-1} - o(1),
  \end{split}
\]
where we used that $n^{v_H-2} p^{e_H-1} = O(\log n)$ whereas $\alpha \ll 1/\log n$.
On the other hand, \Cref{lemma:Delta-p-partial-H} gives that, for some $\lambda = \lambda(H) > 0$ and $C = C(H, \eps)$,
\[
  \begin{split}
    \max_{\cC \in \fC} \Delta_p\big((\partial_e \cH \cup \partial_f \cH)[\ext(\cC)]\big)
    & \le \Delta_p(\partial_e \cH \cup \partial_f \cH) \\
    & \le Cn^{-\lambda} \cdot \left(n^{v_H-2} p^{e_H-1}\right)^2 = o(1).
  \end{split}
\]
Applying Janson's inequality (\Cref{thm:Janson-inequality}), we conclude that
\[
  \xi \le \exp\left(-2\pi_H \cdot (n/r)^{v_H-2} p^{e_H-1} + o(1)\right).
\]
Substituting this estimate into~\eqref{eq:Ex-Z-2-upper} and using \cref{lemma:edges-int-core}, we obtain
\[
  \Ex[Z^2] \le \Ex[Z] + (1+o(1)) \cdot \left(\exp\left(-\pi_H \cdot (n/r)^{v_H-2} p^{e_H-1}\right) \cdot \frac{n^2p}{2r}\right)^2.
\]
Finally, recalling~\eqref{eq:Ex-Z-lower}, in order to get the desired conclusion that $\Ex[Z^2] \le (1+o(1)) \cdot \Ex[Z]^2$, it is enough to argue that
\[
  \exp\left(-\pi_H \cdot (n/r)^{v_H-2} p^{e_H-1}\right) \cdot \frac{n^2p}{2r} \gg 1.
\]
To see that this is the case, note that our upper-bound assumption on $p$ and the assumption that $H$ is $2$-balanced, i.e., $m_2(H) = (e_H-1)/(v_H-2)$, gives
\[
  \pi_H \cdot (n/r)^{v_H-2} p^{e_H-1} \le \pi_H \cdot r^{2-v_H} \cdot (\theta_H-\eps)^{e_H-1} \cdot \log n \le \left(2-\frac{1}{m_2(H)}-\frac{\eps}{\theta_H}\right) \cdot \log n,
\]
whereas our lower-bound assumption on $p$ is that $n^2p \ge \eps n^{2-1/m_2(H)}$.

\appendix
\section{Proof of Proposition \ref{prop:vertices-outside-the-core}}
Let $N \coloneqq \binom{n}{2}$.
  We will show that, for all $n \ll m \le (1 - \delta)N$,
  \begin{equation}
    \label{eq:Ex-xrGi-sum}
    \Ex[x_{r}(\Gnm) + x_{r}(\Gnmp)] = \Omega\left(\sqrt{n^3/m}\right),
  \end{equation}
  which clearly implies the assertion of the proposition.
  As in the proof of \Cref{lemma:eq-d-d-plus-1}, consider the coupling of $G_1 \sim \Gnm$ and $G_2 \sim \Gnmp$ such that $e \in K_n$ is an edge chosen uniformly at random among the non-edges of $G_1$ and the edges of $G_2$.
  Recall from \Cref{fact:eq-analogous-defs} that
  \begin{equation}
    \label{eq:relation-critical-and-not-eq}
    \Pr(e \in \crit_r(G_2)) =  \Pr(e \notin \eq_0^r(G_1)).
  \end{equation}
  Writing $\partial_{G_2}(X_r(G_2))$ for the set of edges of $G_2$ with at least one endpoint in $X_r(G_2)$, we have
  \[
    |\crit_r(G_2)| \ge |\ext(\core_0^r(G_2)) \cap G_2| \ge b_r(G_2) - |\partial_{G_2}(X_r(G_2))|
  \]
  and thus
  \[
    \Pr(e \in \crit_r(G_2)) \ge \Ex\left[\frac{b_r(G_2) - |\partial_{G_2}(X_r(G_2))|}{m+1}\right].
  \]
  On the other hand,
  \[
    \Pr(e \in \eq_0^r(G_1)) \ge \Pr(e \in \int(\core_0^r(G_1))) = \Ex\left[\frac{|\int(\core_0^r(G_1)) \setminus G_1|}{N - m}\right]
  \]
  while
  \[
    \begin{split}
      |\int(\core_0 ^r(G_1)) \setminus G_1|
      & = |\int(\core_0^r(G_1))| - |\int(\core_0^r(G_1)) \cap G_1| \\
      & \ge  |\int(\core_0^r(G_1))| - \frac{m}{r}
        \ge r \cdot \binom{(n-x_r(G_1))/r}{2} - \frac{m}{r} \\
      & = \frac{(n - x_r(G_1))(n - x_r(G_1) - r)}{2r} - \frac{m}{r} \\
      & \ge \frac{N - n \cdot (x_r(G_1) + r) - m}{r},
    \end{split}
  \]
  where the first inequality is true since the edges in $\int(\core_0^r(G_1)) \cap G_1$  do not cross any maximum $r$-cut of $G_1$ and $b_r(G_1) \ge \frac{r-1}{r} \cdot m$ with probability one.
  Hence,
  \[
    \begin{split}
      \Pr(e \notin \eq_0^r(G_1))
      & \le 1- \frac{N - n \cdot (\Ex\left[x_r(G_1)\right] + r) - m}{r \cdot (N-m)} \\
      &= \frac{r-1}{r} + \frac{n \cdot (\Ex\left[x_r(G_1)\right] + r)}{r \cdot (N-m)}.
    \end{split}
  \]

  Substituting the above bounds into \eqref{eq:relation-critical-and-not-eq} and taking expectation, we conclude that
  \[
    \frac{n \cdot (\Ex[x_r(G_1)] + r)}{r \cdot (N-m)} + \frac{\Ex|\partial_{G_2}(X_r(G_2))|}{m+1} \ge \frac{\Ex[b_r(G_2)]}{m+1} - \frac{r-1}{r}.
  \]
  Since $\Ex[b_r(G_2)] = \frac{r-1}{r} \cdot m + \Theta(\sqrt{nm})$ when $n \ll m \le (1-\delta)N$, by \cite[Lemma~14]{C-OMS06},
  \begin{align}
    \label{eq:bound-on-x_r-partial}
    \frac{\Ex[x_r(G_1)]}{n} + \frac{\Ex|\partial_{G_2}(X_r(G_2))|}{m} \ge c \cdot \sqrt{\frac{n}{m}},
  \end{align}
  for some $c > 0$ that depends only on $r$ and $\delta$.
  The following claim bounds the second summand in the left-hand side of~\eqref{eq:bound-on-x_r-partial} from above by an almost-linear function of~$\Ex[x_r(G_2)]$.

  \begin{claim}
    \label{claim:lower-bound-on-partial-X}
    There exists a constant $C$ such that
    \[
      \Ex|\partial_{G_2}(X_r(G_2))| \le \frac{4m}{n} \cdot \Ex[x_r(G_2)] + C \sqrt{\frac{m}{n}} \cdot \Ex[x_r(G_2)] \log \left(\frac{en}			{\Ex[x_r(G_2)]}\right) + 1.
    \]
  \end{claim}

  We first show how \eqref{eq:bound-on-x_r-partial} and \Cref{claim:lower-bound-on-partial-X} imply \eqref{eq:Ex-xrGi-sum}.
  Substituting the upper bound on $\Ex|\partial_{G_2}(X_r(G_2))|$ into \eqref{eq:bound-on-x_r-partial}, we obtain
  \begin{equation}
    \label{eq:ExxrGi-sum-with-log}
    \Ex[x_r(G_1)] + \Ex[x_r(G_2)] + \sqrt{\frac{n}{m}} \cdot \Ex[x_r(G_2)] \log\left(\frac{en}{\Ex[x_r(G_2)]}\right) \ge  \tilde{c} \cdot \sqrt{\frac{n^3}{m}}
  \end{equation}
  for some $\tilde{c}$ that depends only on $r$ and $\delta$.
  Suppose that $\Ex[x_r(G_2)] \le \tilde{c}/3 \cdot \sqrt{n^3/m} \ll n$.
  Since the function $x \mapsto x \log\left(\frac{en}{x}\right)$ is increasing on $(0,n]$, we have
  \[
    \sqrt{\frac{n}{m}} \cdot \Ex[x_r(G_2)] \log\left(\frac{en}{\Ex[x_r(G_2)]}\right)
    \le \frac{\tilde{c}}{3} \cdot \sqrt{\frac{n^3}{m}} \cdot \sqrt{\frac{n}{m}} \log\left(\frac{3e}{\tilde{c}} \cdot \sqrt{\frac{m}{n}}\right)
    \le \frac{\tilde{c}}{3} \cdot \sqrt{\frac{n^3}{m}},
  \]
  where the second inequality holds as $m \gg n$ and $\lim_{x\to 0}x \log(1/x) = 0$.
  This means, by~\eqref{eq:ExxrGi-sum-with-log}, that $\Ex[x_r(G_1)] \ge \tilde{c}/3 \cdot \sqrt{n^3/m}$, as desired.
  We now prove the claim.

  It suffices to show that, with probability at least $1 - 1/n^2$, for every nonempty set $A \subseteq V(G_2)$,
  \begin{equation}
    \label{eq:partial-A-upper}
    |\partial_{G_2}(A)| \le \frac{m}{N} \cdot |A|n + C \sqrt{\frac{m}{n}} \cdot |A|\log\left(\frac{en}{|A|}\right),
  \end{equation}
  where $C$ is a large constant.
  Indeed, this estimate yields
  \[
    \begin{split}
      \Ex|\partial_{G_2}(X_r(G_2))|
      & \le \frac{4m}{n} \cdot \Ex[x_r(G_2)] + C \sqrt{\frac{m}{n}} \cdot \Ex\left[x_r(G_2) \log\left(\frac{en}{x_r(G_2)}\right)\right] + \frac{1}{n^2} \binom{n}{2} \\
      & \le \frac{4m}{n} \cdot \Ex[x_r(G_2)] + C \sqrt{\frac{m}{n}} \cdot \Ex[x_r(G_2)] \log\left(\frac{en}{\Ex[x_r(G_2)]}\right) + 1,
    \end{split}
  \]
  where the second inequality follows from Jensen's inequality applied to the concave function $x \mapsto x \log(en/x)$.

  Note that, for every nonempty set $A \subseteq V(G_2)$, the random variable $|\partial_{G_2}(A)|$ has hypergeometric distribution with mean at most $m/N \cdot |A|n$.
  In particular, standard tail estimates yield (writing $p = m/N$ and $|A| = k$), for every $t \ge 0$,
  \[
    \begin{split}
      \Pr\left(|\partial_{G_2}(A)| \ge \frac{m}{N} \cdot |A|n + t\right)
      &\le \Pr\bigl(\Bin(kn, p) \ge p \cdot kn + t\bigr) \le \exp\left(\frac{-t^2}{2knp + t}\right) \\
      &\le \exp\left(\frac{-t^2}{4knp}\right) + \exp\left(-\frac{t}{2}\right).
    \end{split}
  \]
  Invoking the above estimate with $t \coloneqq C \sqrt{m/n} \cdot |A| \log(en/|A|)$, we conclude that the probability that \eqref{eq:partial-A-upper} fails for some nonempty set $A$ is at most
  \begin{multline*}
    \sum_{k=1}^n \binom{n}{k} \left(\exp\left(-\frac{C^2kN \bigl(\log(en/k)\bigr)^2}{4n^2}\right) + \exp\left(-Ck\log\left(\frac{en}{k}\right)\right)\right) \\
    \le 2 \sum_{k=1}^n \left(\frac{en}{k}\right)^{k-Ck} \le n^{-2},
  \end{multline*}
  where the last inequality is true whenever $C$ is sufficiently large.
\qed

\bibliographystyle{abbrv}
\bibliography{bib.bib}

\begin{thebibliography}{10}

\bibitem{BabSimSpe90}
L.~Babai, M.~Simonovits, and J.~Spencer.
\newblock Extremal subgraphs of random graphs.
\newblock {\em J. Graph Theory}, 14(5):599--622, 1990.

\bibitem{BriPanSte12}
G.~Brightwell, K.~Panagiotou, and A.~Steger.
\newblock Extremal subgraphs of random graphs.
\newblock {\em Random Structures Algorithms}, 41(2):147--178, 2012.

\bibitem{C-OMS06}
A.~Coja-Oghlan, C.~Moore, and V.~Sanwalani.
\newblock M{AX} {$k$}-{CUT} and approximating the chromatic number of random
  graphs.
\newblock {\em Random Structures Algorithms}, 28(3):289--322, 2006.

\bibitem{ConGow16}
D.~Conlon and W.~T. Gowers.
\newblock Combinatorial theorems in sparse random sets.
\newblock {\em Ann. of Math. (2)}, 184(2):367--454, 2016.

\bibitem{CGHS04}
D.~Coppersmith, D.~Gamarnik, M.~T. Hajiaghayi, and G.~B. Sorkin.
\newblock Random {MAX} {SAT}, random {MAX} {CUT}, and their phase transitions.
\newblock {\em Random Structures Algorithms}, 24(4):502--545, 2004.

\bibitem{DeMKah15Tur}
B.~DeMarco and J.~Kahn.
\newblock Tur\'an's theorem for random graphs.
\newblock arXiv:1501.01340.

\bibitem{DeMKah15Man}
B.~DeMarco and J.~Kahn.
\newblock Mantel's theorem for random graphs.
\newblock {\em Random Structures Algorithms}, 47(1):59--72, 2015.

\bibitem{DMS17}
A.~Dembo, A.~Montanari, and S.~Sen.
\newblock Extremal cuts of sparse random graphs.
\newblock {\em Ann. Probab.}, 45(2):1190--1217, 2017.

\bibitem{EngSamWar}
O.~Engelberg, W.~Samotij, and L.~Warnke.
\newblock On the typical structure of graphs not containing a fixed
  vertex-critical subgraph.
\newblock arXiv:2110.10931.

\bibitem{GL18}
D.~Gamarnik and Q.~Li.
\newblock On the max-cut of sparse random graphs.
\newblock {\em Random Structures Algorithms}, 52(2):219--262, 2018.

\bibitem{GKK18}
L.~Gishboliner, M.~Krivelevich, and G.~Kronenberg.
\newblock On {MAXCUT} in strictly supercritical random graphs, and coloring of
  random graphs and random tournaments.
\newblock {\em Random Structures Algorithms}, 52(4):545--559, 2018.

\bibitem{Har60}
T.~E. Harris.
\newblock A lower bound for the critical probability in a certain percolation
  process.
\newblock {\em Proc. Cambridge Philos. Soc.}, 56:13--20, 1960.

\bibitem{Hoe63}
W.~Hoeffding.
\newblock Probability inequalities for sums of bounded random variables.
\newblock {\em J. Amer. Statist. Assoc.}, 58:13--30, 1963.

\bibitem{HosSam}
I.~Hoshen and W.~Samotij.
\newblock Simonovits's theorem in random graphs.
\newblock arXiv:2308.13455.

\bibitem{Jan90}
S.~Janson.
\newblock Poisson approximation for large deviations.
\newblock {\em Random Structures Algorithms}, 1(2):221--229, 1990.

\bibitem{KohLucRod97}
Y.~Kohayakawa, T.~{\L}uczak, and V.~R\"{o}dl.
\newblock On {$K^4$}-free subgraphs of random graphs.
\newblock {\em Combinatorica}, 17(2):173--213, 1997.

\bibitem{Sam14}
W.~Samotij.
\newblock Stability results for random discrete structures.
\newblock {\em Random Structures Algorithms}, 44(3):269--289, 2014.

\bibitem{Sch16}
M.~Schacht.
\newblock Extremal results for random discrete structures.
\newblock {\em Ann. of Math. (2)}, 184(2):333--365, 2016.

\bibitem{Sim68}
M.~Simonovits.
\newblock A method for solving extremal problems in graph theory, stability
  problems.
\newblock In {\em Theory of {G}raphs ({P}roc. {C}olloq., {T}ihany, 1966)},
  pages 279--319. Academic Press, New York-London, 1968.

\bibitem{Tur41}
P.~Tur\'{a}n.
\newblock Eine {E}xtremalaufgabe aus der {G}raphentheorie.
\newblock {\em Mat. Fiz. Lapok}, 48:436--452, 1941.

\end{thebibliography}

\end{document}